\documentclass[12pt]{amsart}

\usepackage{amsfonts, amssymb, amscd}
\usepackage{verbatim}
\usepackage{eucal}
\usepackage{amssymb}
\usepackage{mathrsfs}
\usepackage{graphicx}
\usepackage{psfrag}
\usepackage{cite}
\usepackage{upref}

\def\iso {{\stackrel{\hskip -2pt \sim}\to}}

\def\dual                 {{\vee}}
\def\rk                 {{\rm rk}}
\def\ii                 {{\rm i}}
\def\ee                 {{\rm e}}

\def\Vol		{{\rm Vol}}

\def\ZZ                 {{\mathbb Z}}
\def\PP                {{\mathbb P}} 
\def\RR                 {{\mathbb R}}
\def\CC                 {{\mathbb C}}
\def\QQ                 {{\mathbb Q}}

\def\OO		{{\mathcal O}}

\newtheorem{lemma}{Lemma}[section]
\newtheorem{theorem}[lemma]{Theorem}
\newtheorem{corollary}[lemma]{Corollary}
\newtheorem{proposition}[lemma]{Proposition}
\theoremstyle{definition}
\newtheorem{definition}[lemma]{Definition}
\newtheorem{example}[lemma]{Example}
\newtheorem{conjecture}[lemma]{Conjecture}

\newtheorem{remark}[lemma]{Remark}
\theoremstyle{remark}
\newtheorem*{proof*}{Proof}
\numberwithin{equation}{section}

\title[Applications of HMS: duality conjectures]{
Applications of homological mirror symmetry to hypergeometric systems: duality conjectures}
 
\author{Lev A.  Borisov and R. Paul Horja} 

\address{Department of Mathematics \\ Rutgers University \\ 
Piscataway \\ NJ \\ 08854-8019 \\ USA \\Email: {\tt borisov@math.rutgers.edu}} 

 \address{Universit\"at Wien \\ Fakult\"at f\"ur Mathematik \\Wien, Austria \\
 Email: {\tt richard.paul.horja@univie.ac.at}} 
 
\thanks{The first  
author was partially supported by the NSF grant DMS-1201466.
The second author was supported by the grants
NSF DMS 0854977 FRG, NSF DMS 0600800, NSF DMS 0652633 FRG, NSF DMS
0854977, NSF DMS 0901330, FWF P24572 N25, FWF P20778 and by an ERC Grant.
}

\begin{document}

\begin{abstract}
Homological mirror symmetry for crepant resolutions of Gorenstein toric singularities
leads to a pair of conjectures on certain hypergeometric systems of PDEs. 
We explain these conjectures and verify them in some cases.
\end{abstract}

\maketitle

\section{Introduction} 

Mirror symmetry is the string theory statement that two N=(2,2) superconformal field theories
are isomorphic, up to a so-called mirror involution of the theory. Classically, this is a statement 
about a type IIB theory on one compact Calabi-Yau variety $X$ and a type IIA theory on a mirror 
compact Calabi-Yau variety $Y$ at the level of closed strings. Homological mirror symmetry of Kontsevich is an open-string version of mirror symmetry. It asserts that the derived category 
of coherent sheaves $D^b(X)$ on $X$ (which is the category of boundary conditions on type IIB open strings propagating  in $X$) is equivalent to the derived Fukaya category $DFuk(Y)$ of $Y$, which is a 
remarkable statement that relates two a priori unrelated constructs. 
Since $D^b(X)$ and $DFuk(Y)$ depend non-trivially on the complex parameters of $X$ and on the 
K\"abler parameters of $Y$ respectively, mirror symmetry implies a (at least local) bijection between these sets of parameters.

\medskip

The less well-known but equally fascinating aspect of  homological mirror symmetry
concerns the dependence of the boundary conditions of IIB open strings on $X$ on the 
K\"abler parameters of $X$. Such dependence is locally trivial, but the global structure of this dependence leads to a number
of conjectural consequences. One of them is that birational Calabi-Yau varieties
are derived equivalent, which is a prototypical case of Kawamata's $D$-equivalence conjecture
\cite{Kaw2}.
One also expects to have a locally trivial family of triangulated categories over the base given
by the complex moduli of the mirror Calabi-Yau. After passing to Grothendieck groups this 
should simply be the 
variation of Hodge structures on the cohomology of the mirror
variety.
Our paper deals with this aspect of homological mirror symmetry at the level of 
Grothendieck groups in the important
case of  so-called toric Calabi-Yau
manifolds.

\medskip

Toric Calabi-Yau manifolds are noncompact toric varieties with trivial canonical bundle,
such as crepant resolutions of toric Gorenstein singularities. The lack of compactness
of these Calabi-Yau varieties causes some subtlety in the formulation of homological mirror
symmetry. One also generally needs to use smooth Deligne-Mumford stacks rather than
varieties since higher-dimensional toric singularities often do not admit crepant scheme resolutions. In this paper, we greatly clarify the situation and formulate a set of conjectures
that relate different projective resolutions of a Gorenstein toric singularity. We verify the conjectures 
in a non-trivial example by a brute force calculation and produce other evidence in favor
of the conjectures.

\medskip

To formulate the results of the paper we need to introduce the combinatorial data
that underlie crepant resolutions of toric singularities.
Let $C$ be a rational polyhedral cone in a lattice $N$ which defines an affine toric variety
$X={\rm Spec}\,\CC[C^\dual\cap N^\dual]$. This variety has Gorenstein singularities if and 
only if there exists a linear function  $\deg:N\to\ZZ$ which equals $1$ on the generators of all
one-dimensional rays of $C$. Let $\{v_1,\ldots,v_n\}$ be a subset of the set of lattice points of degree $1$
in $C$ which includes all generators of the rays of $C$. To any projective simplicial 
subdivision $\Sigma$ of $C$ such that the generators of the one-dimensional cones 
are among these $v_i$, one can associate a smooth toric Deligne-Mumford stack $\PP_\Sigma$
with a crepant birational morphism 
$$
\pi:\PP_\Sigma\to X.
$$
To such $\Sigma,$ we associate the Grothendieck group $K_0(\PP_\Sigma)$ 
of $D(\PP_\Sigma)$ as well as the Grothendieck group  $K_0^c(\PP_\Sigma)$ 
of the subcategory of 
complexes whose cohomology sheaves are supported on the compact set $\pi^{-1}(0)$.
We prove that Euler characteristics provides a perfect pairing between these spaces,
see Corollary \ref{perfect}.

\medskip

Crucially, we define isomorphisms between $K_0(\PP_\Sigma)$ 
and $K_0^c(\PP_\Sigma)$ and the solutions 
to so-called better behaved GKZ hypergeometric systems $bbGKZ(C,0)$ and
$bbGKZ(C^\circ,0)$  (see Definition \ref{bbgkz}) respectively given by certain Gamma series
$\Gamma$ and $\Gamma^\circ$.
We claim that these Gamma series are remarkably compatible with the
natural maps 
between the $K_0$ and $K_0^c$ spaces, for different resolutions of singularities of $X$.
Specifically, for two adjacent triangulations $\Sigma_1$ and $\Sigma_2,$ there are 
natural pullback-pushforward maps  $pp$ between the Grothendieck groups which 
we expect to be compatible with the analytic continuation of the solutions of $bbGKZ$.

\medskip

\noindent{\bf Conjecture I (=\ref{conj.ac})}.
The following diagrams of isomorphisms are commutative
$$
\begin{array}{ccc}
K_0(\PP_{\Sigma_2})^\dual&\stackrel{pp^\dual}\longrightarrow& K_0(\PP_{\Sigma_1})^\dual
\\[.3em]
\Gamma\downarrow~{~}&                             &~{~} \downarrow\Gamma
\\[.3em]
bbGKZ(C,0) &\stackrel{a.c.}\longrightarrow&bbGKZ(C,0)
\end{array}
$$
$$
\begin{array}{ccc}
K^c_0(\PP_{\Sigma_2})^\dual&\stackrel{pp^\dual}\longrightarrow& K^c_0(\PP_{\Sigma_1})^\dual
\\[.3em]
\Gamma^\circ\downarrow~{~}&                             & ~{~}\downarrow\Gamma^\circ
\\[.3em]
bbGKZ(C^\circ,0) &\stackrel{a.c.}\longrightarrow&bbGKZ(C^\circ,0)
\end{array}
$$
where the top rows are the duals of the maps induced by the pullback-pushforward
derived functors,
and the bottom rows are analytic continuations along  a certain path in the 
domain of parameters considered in \cite{BH2}.

\medskip

The pairing between $K_0(\PP_\Sigma)$ and $K_0^c(\PP_\Sigma)$ should itself
be given by some conjectural pairing between the solutions of $bbGKZ(C,0)$ and 
$bbGKZ(C^\circ,0)$.

\medskip

\noindent{\bf Conjecture II (=\ref{conj.pairing})}.
There exists a collection $p_{c,d}(x_1, \ldots, x_n)$ of polynomials 
indexed by 
$c\in C$, $d\in C^\circ,$ such that the following hold.
\begin{itemize}
\item All but a finite number of $p_{c,d}$ are zero.

\item For any pair of solutions $(\Phi_c)$ and $(\Psi_d)$ of $bbGKZ(C,0)$ and
$bbGKZ(C^\circ,0)$ respectively the sum
\begin{equation}\tag{$\sharp$}
\sum_{c,d} p_{c,d}\Phi_c\Psi_d
\end{equation}
is constant as a function of $(x_1,\ldots,x_n)$.

\item
The pairing provided by $(\sharp)$ is non-degenerate. 

\item
For any projective simplicial subdivision $\Sigma$, 
the pairing provided by $(\sharp)$ 
is the inverse of the 
Euler characteristics pairing between $K_0(\PP_\Sigma)$ and 
$K^c_0(\PP_\Sigma)$ under the $\Gamma$ and $\Gamma^\circ$ series.
\end{itemize}

\medskip

The paper is organized as follows. In Sections \ref{sec:SR} and 
\ref{sec:K}, we discuss
a combinatorial approach to cohomology with compact support 
(using piecewise polynomial functions) and Grothendieck groups
in the case of semi--projective toric varieties. The technical
details describing the crucial perfect
pairing between the regular Grothendieck group and the
Grothendieck group with compact support are given in 
Section \ref{sec.chi}. The derived category 
origin of the latter Grothendieck group is explained in 
Section \ref{sec:dcat}. Finally, after describing 
the Gamma series solutions to the two classes of
associated better
behaved GKZ systems in Section \ref{gamex}, we 
present the conjectures in Section \ref{sec:conj}.
In the the last two sections of the paper, we  
exhibit evidence in certain classes of examples 
lending strong support for the conjectures. 

\section{Stanley-Reisner calculation for cohomology with  compact support}
\label{sec:SR}
In this preliminary section we describe a Stanley-Reisner type presentation 
of cohomology with compact support of semi-projective toric varieties. While this 
presentation 
is undoubtedly known to some experts in the field, we were unable to find a suitable reference.

\medskip

Let $C$ be a finite rational polyhedral cone\footnote{We will be sloppy in our 
notation in that $C$ will sometimes denote the set
of lattice points of the cone and sometimes the set of real points of it. We hope that this does not lead to confusion.}
in the lattice $N$. We assume that $\dim C = \rk (N)$ but we do not assume $C\cap (-C) = \{0\}$, and 
in fact we allow $C$ to be all of $N$.
Let $v_i$ be a set of $n$ elements of $C$ and let a simplicial fan $\Sigma$ correspond to 
a simplicial complex on the set of indices $i$. We assume that the support of $\Sigma$ is all of $C$.
 We  denote the simplicial
complex (i.e. a set of subsets of the set of indices) by $\Sigma$ as well, as it will be clear from the context
what we consider. Some $v_i$ may not form a cone in $\Sigma$.

\begin{definition}
Denote by 
$A$ the ring of  $\Sigma$-piecewise
polynomial functions on $C$ (with values in $\CC$).
These are collections of polynomial functions on cones $\sigma\in\Sigma$ which are compatible 
with restrictions. Denote by $A^c$  the  ideal of $A$ of functions $f$ with the property 
that $f\vert_{\partial C}=0$.  
\end{definition}

\begin{remark}
When $C$ fills the whole space, we have $\partial C=\emptyset$ and $A^c=A$.
In particular, when $\rk N=0$, we have $A^c=A=\CC$.
\end{remark}

The goal of the rest of the section is to describe natural presentations
of $A$ and $A^c$. The former is well-known, while the latter does not seem to appear in the literature.

\medskip

For each $i$ such that $v_i$ generates a ray of the fan, 
consider a piecewise linear function $D_i$ on $C$ which takes the 
value $1$ on $v_i$ and $0$ on the other
generators of rays of $\Sigma.$ The piecewise linear function $D_i$
is called the Courant function associated to $v_i,$ see \cite{Billera}.
If $v_i$ is not a ray generator of $\Sigma$, 
set $D_i=0$. The following is the standard Stanley-Reisner presentation of $A$, see 
\cite[2.3, 3.6]{Billera}, \cite[1.3]{Brion}.
\begin{proposition}
The ring $A$ is naturally isomorphic to the quotient of $\CC[D_1,\ldots,D_n]$ by the ideal generated by the monomials
$$
\prod_{i\in I} D_i
$$
for all $I\not\in \Sigma$.
\end{proposition}

We now give an analogous statement for $A^c$.
\begin{proposition}
The module $A^c$ over $\CC[D_1,\ldots,D_n]$ has the following natural presentation
using generators $F_I,$ for $I \in \Sigma$ with $\sigma_I^\circ\subseteq C^\circ.$
It is the quotient of the free module
$$
\bigoplus_{I\in\Sigma, \sigma_I^\circ\subseteq C^\circ} \CC[D_1,\ldots,D_n] F_I
$$
by the submodule generated by relations
\begin{equation}\label{1}
D_i F_I  - F_{I\cup \{i\}},~{\rm for}~i\not\in I, I\cup\{i\}\in \Sigma
\end{equation}
and 
\begin{equation}\label{2}
D_i F_I  ,~{\rm for}~i\not\in I, I\cup\{i\} \notin \Sigma.
\end{equation}
\end{proposition}

\medskip
\begin{proof}
Note that there is a natural morphism of $A$--modules from
\begin{equation}\label{map}
\bigoplus_{I\in\Sigma, \sigma_I^\circ\subseteq C^\circ} \CC[D_1,\ldots,D_n] F_I/\langle{\rm relations~\eqref{1}\eqref{2}}\rangle
\end{equation}
to $A^c$ which sends 
$F_I\mapsto \prod_{i\in I}D_i$.

\medskip

Further observe that the rings $\CC[D_1,\ldots,D_n]$, $A$ and the module $A^c$ are naturally $\ZZ^{n}$ graded
where $n$ is the number of indices $i$. The underlying reason for such grading is the $(\RR_{>0})^n$ action on $C$ that scales the rays of the fan and scales the rest of $C$ linearly with it. This
has an effect of separately scaling $D_i$. 

\medskip

The support of a monomial $\prod_i D_i^{k_i}$ is the set of 
indices $i$ such that $k_i >0.$
The non-zero monomials that occur in $A$ are exactly
the ones whose support lies in $\Sigma$.
Clearly, $A^c$ is a $\ZZ^{n}$ graded ideal of $A$. A non-zero 
monomial $\prod_i D_i^{k_i}$ lies in it iff
its support lies in $\Sigma$ but is not a boundary simplex of $C$. Thus every monomial $m$ in $A^c$ 
is divisible by the image of $F_{{\rm supp}(m)}$. This shows that the map \eqref{map} 
is surjective, i.e. the images of $F_I$ generate $A^c$. 

\medskip
Any monomial ideal has a resolution whose first step is
given by the lcm of generators, the Taylor's resolution of \cite[17.11]{Eisenbud}. 
Thus, it suffices to show that for any $I$, $J$
with $\sigma_I^\circ,\sigma_J^\circ\subseteq C^\circ$ the corresponding
relation
$$
(\prod_{i\in I\backslash J} D_i)F_J-
(\prod_{i\in J\backslash I} D_i)F_I
$$
lies in the module generated by \eqref{1} and \eqref{2}. 
If $I\cup J\not\in\Sigma$, then it is easy to see that both terms lie in the module generated by \eqref{2}.
Otherwise, all of the intermediate sets $J\subseteq I_1\subseteq I\cup J$ have the property 
$\sigma_{I_1}^\circ\subseteq C^\circ$, and the first  term is equal to $F_{I\cup J}$ modulo \eqref{1}.
The same is true for the second term, which finishes the proof.
\end{proof}

\begin{remark}
We will now observe that $A^c$ is a dualizing module for $A$. Recall that $A$ is Cohen-Macaulay,
a fact that goes back to the works of Reisner, Stanley and Hochster (but see 
\cite[1.3]{Brion} for an approach closer to the spirit of this paper). 
Let $\partial \Sigma$ denote the subfan as well as the simplicial
subcomplex of $\Sigma$ consisting of sub-cones/simplices supported on the 
boundary $\partial C.$ Topologically, the support of the simplex 
$\partial \Sigma$ is a sphere, so
a result of Munkres \cite[Theorem 2.1]{Munkres} shows that it is
Gorenstein over $\CC.$ According to a theorem of Hochster
\cite[Theorem 5.7.2]{BruHer}, this means that
the dualizing module of the Stanley--Reisner 
ring $A$ is isomorphic as a $\ZZ^n$-graded $A$--module
to the graded ideal of $A$ generated
by $\prod_{i \in I} D_i$ with $I \in \Sigma \setminus \partial \Sigma.$
But these are exactly the simplices $I \in \Sigma$ such that
$\sigma_I^\circ\subseteq C^\circ,$ and the proof of 
the previous proposition provides the claim.
\end{remark}

\medskip

Consider the grading on $A$ and $A^c$ with $\deg D_i=1$ and $\deg F_I = \vert I\vert$.
Global linear functions on $N$ form a dimension $\rk N$ subspace in the degree one component of $A$. 
They are given in terms of the generators by 
$$
\sum_i (m\cdot v_i)D_i,~m\in N^\dual
$$
and form a regular sequence for any basis of $N^\dual$. Denote by $Z$ the ideal they generate. Then $A/ZA$ is a graded Artinian ring,
and $A^c/ZA^c$ is its dualizing module. Thus, $A^c/ZA^c$ has one-dimensional socle and the pairing
induced by multiplication and evaluation at the socle is nondegenerate. 
We can explicitly calculate it as follows.
\begin{proposition}\label{int}
For each $I\in\Sigma$ with $\vert I\vert =\rk N$ we define $\Vol_I$  to be the index of the sublattice in $N$ generated
by $v_i,i\in I$.
There exists a unique linear function  $\int :A^c/ZA^c\to \CC$ that takes values $\frac 1{\Vol_{I}}$ on $F_I$ for all $I\in \Sigma$ with $\vert I\vert =\rk N$.
This function is nonzero on the socle element of $A^c/ZA^c$ and induces the pairing.
\end{proposition}

\begin{proof}
The function $\int$ is defined as follows, see \cite{Brion0}. For a continuous piecewise polynomial function $f=(f_J)$ consider a rational 
function on $N_\CC$ given by 
\begin{equation}\label{this}
\sum_{\vert J\vert = \rk N} \frac 1{\vert {\rm Vol}(J)\vert} f_J \prod_{i\in J} \frac 1{u_{i,J}}
\end{equation}
where $u_{i,J}$ is a linear function on $N$ with value $1$ at $v_i$ and $0$ on other $v_j\in J$ 
(in other words, the restriction of $D_i$ to the cone $J$).
The rational function \eqref{this} is in fact polynomial. Indeed, the singularities of individual terms occur along the interior walls of $\Sigma$. However,
for every such wall the two adjacent terms cancel each other. Thus the resulting function has singularities at (complex) codimension two 
only and is therefore nonsingular.
Then 
$$\int f =\left( \sum_{\vert J\vert = \rk N} \frac 1{\vert {\rm Vol}(J)\vert} f_J \prod_{i\in J} \frac 1{u_{i,J}} \right)(0)
$$
provides a map $\int:A^c \to \CC$. 

\medskip

We will now verify the  properties of $\int$. The above argument shows that for a global polynomial function $g$ we have 
$$
\int fg = g(0)\int f.
$$
In particular, $\int$ passes through to the map $A^c/ZA^c\to \CC$.
We can also easily see that $\int$ vanishes on all but the top degree component of $A^c/ZA^c$.
Indeed, for $\lambda\in \CC^*$ consider the scaling $\lambda: N_\CC\to N_\CC$. We have
$$
\int f\circ \lambda = \lambda^{\rk N} \int f
$$
because of the scaling of $u_{i,j}$ in \eqref{this}. Thus $\int f =0$ for $f$ in all but the top eigenspace
of $\CC^*$ action. It remains to observe that for $\vert I\vert=\rk N$,  $\int F_I = \frac 1{\vert {\rm Vol}(I)\vert} $ because the function in \eqref{this} is constant.
\end{proof}

\section{$K$-theory}\label{sec:K}
In this section we consider more sophisticated combinatorial objects associated to 
$C$, $N$ and $\Sigma$. The first one is the Grothendieck group of the corresponding toric
DM stack $\PP_\Sigma$. The second one is new and will be later shown to be some
kind of Grothendieck group with compact support. At the end of the section we 
will introduce an example, to which we will be returning frequently throughout the rest of the paper.

\medskip

We start with what is by now a standard definition of the toric stack associated to $C$, $N$, $v_i$ 
and $\Sigma$, see \cite{BCS}. 
\begin{definition}
Consider the open subset $U$ of $\CC^n$ given by 
$$
U := \{(z_1,\ldots,z_n),~{\rm such~that~} \{i, z_i=0\}\in \Sigma\}.
$$
Define the group $G$ as the subgroup of $(\CC^*)^n$ given by
$$
G=\{(\lambda_1,\ldots,\lambda_n),~{\rm such~that~}\prod_{i=1}^n \lambda_i^{\langle m,v_i\rangle}=1{\rm ~for~all~}m\in N^\dual\}.
$$
Define $\PP_\Sigma$ to be the \emph{stack} quotient of $U$ by $G$.
\end{definition}

\begin{remark}
The notation $\PP_\Sigma$ is often used  to denote the GIT quotient $U/G$ which 
is typically a singular toric variety. However,
the reader can be assured that we always work on the smooth toric DM stack.
\end{remark}

\medskip

The Grothendieck group of $\PP_\Sigma$ (equivalently, the Grothendieck group of 
$G$-equivariant coherent sheaves on $U$) has been calculated in 
\cite{BH1}.
\begin{proposition}
Let $C$, $v_i$ and $\Sigma$ be as before. Consider the smooth toric DM stack $\PP_\Sigma$  defined by $\Sigma$
and $v_i$. Then the $K$-theory  $K_0(\PP_\Sigma)$ (with complex coefficients) is the quotient of the ring $\CC[R_i^{\pm 1}]$ by the 
relations
\begin{equation}\label{Kthry}
\prod_{i=1}^n R_i^{m\cdot v_i} - 1,~m\in N^\dual {\emph ~~and~~~}~
\prod_{i\in I}(1-R_i),~ I\not\in\Sigma.
\end{equation}
\end{proposition}

\medskip

We will now describe the structure of $K_0(\PP_\Sigma)$ in more detail in combinatorial terms.
The ring $K_0(\PP_\Sigma)$ is semilocal, with the local summands given by so-called twisted sectors, see the
definition below. 
\begin{definition}
For every cone $\sigma\in\Sigma$ we denote by ${\rm Box}(\Sigma)$ the set of points of $\gamma\in N$ that can be 
written as $\gamma=\sum_{i\in\sigma} \gamma_iv_i$ with $0\leq \gamma_i < 1$. We denote by ${\rm Box}(\Sigma)$ 
the union of ${\rm Box}(\sigma)$ for all $\sigma\in\Sigma$. 
\end{definition}

The elements of ${\rm Box}(\Sigma)$ are in one-to-one
correspondence with components of inertia stack of $\PP_\Sigma$. In the following definition
we will introduce the cohomology of these components.

\begin{definition}
To each $\gamma\in {\rm Box}(\Sigma)$ we associate a toric 
stack called {\it the twisted sector} defined as the closed toric substack associated to the minimum cone in $\Sigma$ 
that contains $\gamma$. We denote the corresponding rings (ideals) of piecewise polynomial functions (vanishing 
on the boundary) by $A_\gamma$ ($A_\gamma^c$). We denote the corresponding Artinian rings and modules by $H_\gamma $ and $H_\gamma^c$ respectively.
\end{definition}

\begin{remark}
More specifically, for $\gamma\in {\rm Box}(\Sigma)$ we define $\sigma(\gamma)$ to be the minimum
cone of $\Sigma$ that contains $\gamma$. We define $N_\gamma = N/{\rm Span}(\sigma(\gamma))$
to be quotient lattice and $\Sigma_\gamma$ be the quotient fan
\footnote{More accurately, one may consider the quotient by the span of $v_i$ in $\sigma(\gamma)$.
The resulting lattice has torsion and leads to toric DM stacks which are not schemes at generic point,
see \cite{BCS}. This leads to various factors $\frac 1{\vert {\rm Box}(\sigma(\gamma))\vert}$
in the formulas of the paper.}. It is the image of the star of $\sigma(\gamma)$ and consists of $I-\sigma(\gamma)$ for all $I$ in $\Sigma$ with  $I\supseteq \sigma(\gamma)$. The map 
from $N$ to $N_\gamma$ is denoted by\, $\bar {}$. We also use $\bar D_i$ and $\bar F_I$ 
for the corresponding elements of $A_\gamma, H_\gamma, A^c_\gamma, H^c_\gamma$.
\end{remark}

\medskip

The key to understanding $K_0(\PP_\Sigma)$ is provided by the following result.
\begin{proposition}\cite{BH}\label{prop:ch}
There is a natural algebra isomorphism
$$ch: K_0(\PP_\Sigma)\iso  \bigoplus_{\gamma\in {\rm Box}(\Sigma)} H_\gamma.$$
\end{proposition}

Let us recall the construction of this isomorphism $ch$, since it will be useful later. One can show that the algebra homomorphisms
from $K_0(\PP_\Sigma)$ to $\CC$ correspond to $\gamma=\sum_{i\in\sigma(\gamma)}\gamma_i v_i$ in ${\rm Box}(\Sigma)$ by
$R_i\mapsto \ee^{2\pi\ii \gamma_i}$. The cohomology 
$H_\gamma=A_\gamma/Z_\gamma A_\gamma$ of the corresponding twisted sector is generated by $\bar D_i$ for $i\in {\rm Star}(\sigma(\gamma))-\sigma(\gamma)$ with the 
relations 
$$
\prod_{i\in J} \bar D_i = 0
$$
for $J$ not in a cone in $ {\rm Star}(\sigma(\gamma)),$ and
$$\sum_{i\in {\rm Star}(\sigma(\gamma))-\sigma(\gamma)} (m\cdot v_i) \bar D_i=0$$
for $m\in {\rm Ann}(v_i,i\in \sigma(\gamma))$.
The projection ${\rm ch}_\gamma:K_0(\PP_\Sigma)\to A_\gamma/Z_\gamma A_\gamma$ is given by 
\begin{equation}\label{ch}
\begin{array}{ll}
{\rm ch}_\gamma(R_i) = 1, &i\not\in  {\rm Star}(\sigma(\gamma))\\
{\rm ch}_\gamma(R_i) = \ee^{\bar D_i},& i\in {\rm Star}(\sigma(\gamma))-\sigma(\gamma)\\
{\rm ch}_\gamma(R_i) = \ee^{2\pi\ii \gamma_i }\prod_{j\not\in \sigma(\gamma)} {\rm ch}_\gamma(R_j)^{(m_i\cdot v_j)},& i\in \sigma(\gamma)
\end{array}
\end{equation}
where $m_i$ is any $\QQ$-valued linear function on $N$ which takes values $-1$ on $v_i$ and $0$ on all other $v_j, j\in \sigma(\gamma)$. 
Here, in the last line, the rational powers of unipotent elements are well-defined.

\begin{remark} 
The first line in \eqref{ch} can be thought of as a particular case of the second line, under $\bar D_i=0$ for $i$ not in the 
induced fan. 
\end{remark}

\medskip

We now define a module $K_0^c(\PP_\Sigma)$ over $K_0(\PP_\Sigma)$ which will turn out to be isomorphic to the sum of cohomologies with compact support 
for all of the twisted sectors. 
\begin{definition}
Consider the module $K_0^c(\PP_\Sigma)$ over the ring $K_0(\PP_\Sigma)$ generated by $G_I$, $I\in\Sigma,
\sigma_I^\circ \subseteq C^\circ$, with relations for all $i$, $I$ with $i\not\in I$
\begin{equation}\label{K1}
(1-R_i^{-1})G_I = G_{I\cup \{i\}},~{\rm if~} I\cup \{i\}\in\Sigma,
\end{equation}
\begin{equation}\label{K2}
(1-R_i^{-1})G_I = 0, ~{\rm if~} I\cup \{i\}\notin\Sigma.
\end{equation}
\end{definition}

\begin{remark} We will later see that $K^c_0(\PP_\Sigma)$ is isomorphic to the Grothendieck group  
of a certain category of "compactly supported" sheaves on $\PP_\Sigma$, which 
will justify the notation.
\end{remark}

\begin{proposition}\label{prop:chc}
There is a natural isomorphism
$$ch^c:K_0^c(\PP_\Sigma)\iso
\bigoplus_{\gamma\in{\rm Box}(\Sigma) }H_\gamma^c$$
 compatible with the ring isomorphism $ch:K_0(\PP_\Sigma)\iso
\bigoplus_{\gamma\in{\rm Box}(\Sigma) }H_\gamma$.
\end{proposition}

\begin{proof}
We define $ch^c=\bigoplus_{\gamma}ch_\gamma^c$
where $ch_\gamma^c:K_0^c(\PP_\Sigma)\to H_\gamma^c$
is given by
$$
ch_v^c(\prod_{i=1}^n R_i^{l_i}G_I)=0$$
for $I\not\subseteq {\rm Star}(\sigma(v))$
and 
\begin{equation}\label{chcgamma}
\begin{array}{ll}
ch_\gamma^c(\prod_{i=1}^n R_i^{l_i}G_I)=&
 \prod_{i=1}^n ch_\gamma(R_i)^{l_i} \prod_{i\in I, i\not\in\sigma(\gamma)}\Big( \frac {1-e^{- \bar D_i}}{\bar D_i}\Big) \\
 &
 \cdot \prod_{i\in I\cap\sigma(\gamma)}(1-ch_\gamma(R_i)^{-1})\bar F_{\bar I} 
\end{array}
\end{equation}
for $I\subseteq {\rm Star}(\sigma(\gamma))$. The cone $\bar I$ 
in the induced fan is defined by the set of indices in $I$, but not in $\sigma(\gamma)$. 
The $\bar F_{\bar I}$ indicates the generator of $H_v^c$   that corresponds 
to $\bar I$ in the induced fan $\Sigma_\gamma$.

\medskip

We need to prove that the above $ch^c$ satisfies the claim of the proposition.
First of all, we need to prove that $ch_\gamma^c$ is well-defined for any $\gamma\in{\rm Box}(\Sigma)$.
Because $ch_\gamma:K_0(\PP_\Sigma)\to H_\gamma$ is a ring homomorphism, we only need to check
that the above prescription agrees on two sides of  \eqref{K1} and \eqref{K2}. 
If $I\not\subseteq {\rm Star}(\sigma(\gamma))$,
then both sides map to $0$. If $I\subseteq{\rm Star}(\sigma(\gamma))$, but 
$I\cup\{i\}\not\subseteq {\rm Star}(\sigma(\gamma))$, then the right hand side is always sent to $0$ by the 
above prescription. In this case $i\not\in {\rm Star}(\sigma(\gamma))$, thus $ch_\gamma(R_i)^{-1}=1,$ so the 
left hand side is trivially sent to $0$. We are left to consider the case $I\cup\{i\}\subseteq {\rm Star}(\sigma(\gamma))$.
If $I\cup\{i\}\not\in\Sigma$, then  $i\not\in\sigma(\gamma)$. The left hand side of \eqref{K2} 
goes to zero because $(1-ch_\gamma(R_i)^{-1}) = (1-\ee^{-\bar D_i}) \sim \bar D_i $ and $\bar D_i{\bar F}_{\bar I} =0$
in $H_\gamma^c$. If $I\cup\{i\}\in\Sigma$ (and in ${\rm Star}(\sigma(\gamma))$), then
there are two possibilities. If $i\not\in\sigma(v),$
the statement follows from $\bar D_i{\bar F}_{\bar I} = \bar F_{\overline {I\cup\{i\}}}$. If $i\in \sigma(\gamma)$, then
$\bar I = \overline{I\cup\{i\}}$ and the statement is immediate. 

\medskip

To show that the map $ch^c$ is an isomorphism, observe that every finitely generated module 
over an Artinian ring is a direct sum of its localizations at maximal ideals. In our case, the 
maximum ideals correspond to $\gamma\in{\rm Box}(\Sigma)$, so it suffices to show that the
induced map 
$$
ch_\gamma^c :( K_0^c(\PP_\Sigma) )_\gamma\to H_\gamma^c
$$
is an isomorphism.
The localization with respect to this maximal ideal is characterized by the
nilpotency of $R_i-\ee^{2\pi\ii \gamma_i}$ where $\gamma=\sum_{i\in\sigma(\gamma)}\gamma_i v_i,$
and $\gamma_i=0$ for $i\not\in\sigma(\gamma)$.  The localization of $K_0^c(\PP_\Sigma)$ is generated 
as a module over the localization of $K_0(\PP_\Sigma)$ (that we know to be isomorphic to $H_v$) by 
$G_I$ with the relations \eqref{K1} and \eqref{K2}. Observe that if $\lambda_i\neq 0$, then $(1-R_i^{-1})$ is invertible, 
thus \eqref{K2} implies
that $G_I=0$ for $I\not\ni i$. Consequently, we may restrict our attention to $I\subseteq {\rm Star}(\sigma(\gamma))$.
 Note that we are only using $I$
such that $\sigma_I^\circ \subseteq C^\circ$, and the images modulo the 
span of $\sigma(\gamma)$
have this property. Moreover, $(1-R_i^{-1})$ are invertible for $i\in\sigma(\gamma)$, 
which allows
to pass from $I$ to $I-{\rm \sigma(\gamma)}$. It remains to observe that for $I\subseteq {\rm Star}(\sigma(\gamma))-\sigma(\gamma)$
the  relations \eqref{K1} and \eqref{K2} become
the defining relations on $H_\gamma^c$ under $ch_\gamma$ and 
$G_I\to  \prod_{i\in I}\Big( \frac {1-e^{- \bar D_i}}{\bar D_i}\Big) 
\bar F_{\bar I}$.
\end{proof}

\medskip

\begin{remark}
The choice of $ch^c$ is also compatible with the natural maps  $H^c_\gamma\to H_\gamma$
and $K_0^c(\PP_\Sigma)\to K_0(\PP_\Sigma)$. We will not use this fact in the paper,
but it is an important motivation behind the definition of $ch^c$.
\end{remark}

\medskip

As promised, we now introduce the key example that will be featured prominently 
throughout the paper.
\begin{example}\label{keyex}
Let $C\subset \ZZ^2$ be the cone generated by $v_1=(0,1)$, $v_2=(1,1)$ and $v_3=(3,1)$.
Let $\Sigma$ be the simplicial complex with the maximum sets $\{1,2\}$ and $\{2,3\}$.
Let us first construct $\PP_\Sigma$. It is the stack quotient $[U/G]$ where 
$$
U=\{(z_1,z_2,z_3),\,(z_1,z_3)\neq (0,0)\},~~G=\{(\lambda^2,\lambda^{-3},\lambda),\,\lambda\in \CC^*\}.$$

\medskip\noindent
Let us now consider the Grothendieck groups of $\PP_\Sigma$. 
The group $K_0(\PP_\Sigma)$ is written as 
$$
\begin{array}{ll}
&\CC[R_1^{\pm 1},R_2^{\pm 1},R_3^{\pm 1}]/
\langle (1-R_1)(1-R_3), R_1R_2R_3-1,R_2R_3^3-1\rangle
\\
&=\CC[R_3^{\pm 1}]/\langle (1-R_3)^2(1+R_3)\rangle
=\CC\oplus \CC R_3 \oplus \CC R_3^2,
\end{array}
$$
and
$$
K_0^c(\PP_\Sigma)=K_0(\PP_\Sigma)G_2
=\CC G_2\oplus \CC R_3G_2 \oplus \CC R_3^2G_2.
$$

\medskip\noindent
There are two twisted sectors, which correspond to $\gamma=(0,0)$ and 
$\gamma=(2,1)=\frac 12 v_2+\frac 12 v_3$. The $(0,0)$ sector is sometimes called the
untwisted sector. \footnote{By a quirk of terminology, we 
consider the untwisted sector to be one element of the larger set of 
twisted sectors.} We have 
$$
H_{(0,0)} = \CC[D_1,D_2,D_3]/\langle D_1,D_3, D_2+3D_3,D_1+D_2+D_3\rangle
=\CC\oplus \CC D_3
$$
$$
H_{(2,1)} = \CC[\emptyset] = \CC.
$$
The cohomology with compact support $H^c_{(0,0)}$ is generated by $F_2$ and we have 
$$
H^c_{(0,0)}=H_{(0,0)}F_2 = \CC F_2 \oplus \CC D_3 F_2,~~
H^c_{(2,1)}=\CC \bar F_{\emptyset}.
$$
The integration maps $\int$ are given by 
$$
\int_{(0,0)} F_2 = 0,~\int_{(0,0)}D_3F_2 = \frac 12,~\int_{(2,1)}\bar F_{\emptyset} = 1.
$$

\medskip\noindent
The maps $ch$ and $ch^c$ are given by
\begin{equation}\label{chmaps}
\begin{array}{l}
ch(1) = 1\oplus 1,~ch(R_3) = (1+D_3)\oplus (-1),~ch(R_3^2) = (1+2D_3)\oplus 1,\\
ch^c(G_2) = (1+\frac 32 D_3)F_2\oplus  2 \bar F_\emptyset,
~
ch^c(R_3G_2) = (1+\frac 52 D_3 )F_2\oplus (-2) \bar F_\emptyset,\\
ch^c(R_3^2G_2) = (1+\frac 72 D_3 )F_2\oplus 2 \bar F_\emptyset.
\end{array}
\end{equation}
\end{example}

\section{Euler characteristics}\label{sec.chi}
In this section we define the pairing between $K_0(\PP_\Sigma)$ and $K_0^c(\PP_\Sigma)$.
It is based on the linear map 
$$\chi:K_0^c(\PP_\Sigma)\to \CC$$
which will later be reinterpreted as the 
Euler characteristic of some sheaves on $\PP_\Sigma$.
We then calculate it for the Example \ref{keyex}.

\medskip

For $\alpha\in \ZZ^n$ and $G_I$ we
define $\chi(\prod_i R_i^{\alpha_i}G_I)$ as follows. Let $M=N^\dual$ be the dual lattice.
Consider the following element in the field of rational functions $\CC(M_\QQ)$
\begin{equation}\label{chieq}
\sum_{\stackrel{J\supseteq I, }{|J|=\rk N}}
\frac 1{\vert {\rm Box}(J)\vert}
\sum_{\stackrel{\gamma \in {\rm Box}(J)}{\gamma=\sum_{j\in J}\gamma_jv_j}}
\prod_{i\in J} q^{-u_{i,J}\alpha_i}\ee^{-2\pi\ii \gamma_i\alpha_i}
\prod_{i\in J-I} \frac 1{1-q^{u_{i,J}}\ee^{2\pi\ii\gamma_i}}
\end{equation}
where $u_{i,J}$ form a dual basis to $v_i, i\in J$. The sum is a polynomial in $\CC[M_\QQ]$,
rather than
a rational function in $\CC[M_\QQ]$. We then evaluate it at $q=0$ to get the 
the value of $\chi$. We will later see that this is an integer.

\medskip
Our first goal is prove the key properties of $\chi$, starting from the fact that it is well-defined.
\begin{proposition}\label{makessense}
The above definition of $\chi(\prod_i R_i^{\alpha_i}G_I)$ makes sense.
\end{proposition}

\begin{proof}
We need to show that the resulting function is a Laurent polynomial with fractional exponents in
$\CC[M_\QQ]$. It can be easily seen to be the Euler characteristics of the line bundle
$\bigotimes_i {\mathcal L}_i^{\alpha_i}$ on the smooth toric DM stack associated to $I$. More precisely, 
it calculates this Euler characteristics  as the Euler characteristics of the pushforward of this bundle to $\PP_\Sigma$.
Recall that $\PP_\Sigma$ is the stack quotient $[U/G]$ for the following $U$ and $G$. 
The scheme $U$ is an open subscheme of $\CC^n$ whose closed points 
 $(z_1,\ldots,z_n)$ have the property that there exists a subset in $\Sigma$ that contains all 
of the indices for which $z_i=0$. The group $G$ is the subgroup of $(\CC^*)^n$ that acts 
diagonally on $\CC^n$ described by $\lambda=(\lambda_1,\ldots, \lambda_n),~\prod_{i=1}^n\lambda_i^{m\cdot v_i}=1$
for all $m\in M$. The aforementioned pushforward corresponds to the $G$-equivariant module on
$\CC^n$ which is isomorphic to 
$$F=\CC[z_1,\ldots,z_n]/\langle z_i,i\in I\rangle$$ 
with the linearization
$\lambda^*\prod_jz_j^{r_j} = \prod_j\lambda_j^{r_j-\alpha_j}z_j^{r_j}$. 

\medskip

We consider the \v{C}ech cover of $U$ by $U_{J},J\in \Sigma,$ defined as the subsets with
$z_j\neq 0, j\not\in J$. We calculate the equivariant cohomology
of $F$ by the complex of invariant sections. If $J\supsetneq I$, 
then the observe that the sections are zero. We apply the delta function trick of \cite{BorLibg}
to show that only the contributions of maximum-dimensional $J$ appear. Finally, to calculate 
the graded dimensions of the contribution of such $U_J$ we need to account for the sublattice,
which is where the summation over ${\rm Box}(J)$ comes in.

\medskip

The resulting Euler characteristics is finite, because of projectivity of the corresponding twisted sector,
which is assured by the assumption on $I$.
\end{proof}

\begin{remark}
An alternative proof that $\chi$ is well-defined follows from the argument of 
Proposition \ref{hrr}.
\end{remark}

\medskip
\begin{proposition}
The linear function $\chi$ descends to a map $K_0^c(\PP_\Sigma)\to \CC$.
\end{proposition}

\begin{proof}
We first observe that for any $m\in M$
$$
\chi(\prod_i R_i^{l_i}G_I) = \chi(\prod_i R_i^{l_i+m\cdot v_i}G_I)
$$
because the expression \eqref{chieq} is changed by $q^{m}$. The
analogous statement for \eqref{K2} can be observed directly from \eqref{chieq}
or from the Koszul complexes and the geometric description of $\chi$.
\end{proof}

We now define a pairing between $K_0(\PP_\Sigma)$ and $K_0^c(\PP_\Sigma)$.
\begin{definition}\label{Epairing}
We define the Euler characteristics pairing
$$
\chi:K\times K^c\to \ZZ
$$
by 
$$\chi(\prod_i R_i^{\beta_i},\prod_i R_i^{\alpha_i}G_I)=
\chi(\prod_i R_i^{\alpha_i-\beta_i}G_I).
$$
\end{definition}

We now would like to compare the pairing between $K_0(\PP_\Sigma)$ and $K_0^c(\PP_\Sigma)$
with the pairings between $H_\gamma$ and $H_\gamma^c$ for the twisted sectors. The key to 
this is to view
$\chi$ as a linear function on $\bigoplus_{\gamma\in{\rm Box}(\Sigma)}H_\gamma^c$, which is essentially a combinatorial statement of a version of the
Hirzebruch-Riemann-Roch theorem.

\medskip
Recall that we have defined maps $\int_\gamma:H^c_\gamma\to \CC$.
\begin{proposition}\label{hrr}
For an arbitrary $v\in K_0^c(\PP_\Sigma)$  we have 
$$
\chi(v) = 
\sum_\gamma \frac 1{\vert {\rm Box}(\sigma(\gamma))\vert}
\int_\gamma \prod_{i\in \sigma(\gamma)} \frac 1{(1-ch(R_i^{-1}))}
$$
$$
\prod_
{\stackrel {i\in{\rm Star}(\sigma(\gamma))}{i\not\in \sigma(\gamma)}}\Big(\frac{D_{\bar i}}{1-ch(R_i^{-1})}\Big) ch_\gamma^c(v).
$$
\end{proposition}

\begin{proof} 
Consider $v=(\prod_i R_i^{\alpha_i})G_I$. 
Then we have from \eqref{chieq} that
$\chi(v)$ is given by the evaluation at 
${\bf u}=0$ of the expression
$$
\sum_{\stackrel{J\supseteq I, }{|J|=\rk N}}
\frac 1{\vert {\rm Box}(J)\vert}
\sum_{\stackrel{\gamma \in {\rm Box}(J)}{\gamma=\sum_{j\in J}\gamma_jv_j}}
\prod_{i\in J} \ee^{\alpha_iu_{i,J}+2\pi\ii \gamma_i\alpha_i}
\prod_{i\in J-I} \frac 1{1-\ee^{-u_{i,J}-2\pi\ii\gamma_i}}
$$
where $u_{i,J}$ is a linear function on $N_\CC$ with values $1$ at $i$ and $0$ at 
other elements of $J$. Note that we have changed the sign in ${\bf u}$, which will not matter
since we are evaluating at $0$. We also switched from $\gamma$ to a dual twisted sector with
$\ee^{2\pi\ii\gamma_i^\dual} = \ee^{-2\pi\ii\gamma_i}$.
Let us rewrite this as a summation over the twisted sectors, i.e. over the
elements of ${\rm Box}(\Sigma)$.
$$
\chi(v)=\sum_{\gamma=\sum_{i\in \sigma(\gamma)}\gamma_iv_i\in {\rm Box}(\Sigma)}\chi_\gamma(v),
$$
where 
$\chi_\gamma(v)$ is the evaluation at ${\bf u}=0$ of
$$
\sum_{J\supseteq \sigma(\gamma)\cup I, |J|=\rk N}
\frac 1{\vert {\rm Box}(J)\vert}
{\prod_{i\in J} \ee^{\alpha_iu_{i,J}+2\pi\ii \gamma_i\alpha_i}}
\prod_{i\in J-I} \frac 1{1-\ee^{-u_{i,J}-2\pi\ii\gamma_i}}.
$$
Note that $\chi_\gamma(v)$ is well-defined. Indeed, in the neighborhood of ${\bf u}=0$ the
terms of the summation
have poles of order one along divisors $u_{i,J}=0$ for some $J\supseteq I\cup \sigma(\gamma)$
with $i\not\in I\cup\sigma(\gamma)$. Each such occurrence corresponds to a codimension one 
cone in $\sigma$ given by $J-\{i\}$. This cone contains $I$ and thus an interior point of $C$.
Thus, each cone appears in two terms, for adjacent $J_1$ and $J_2$ with 
$J_1-\{i_1\}=J_2-\{i_2\}$. The corresponding
contributions are the same at a generic point of $u_{i,J}=0$, except for 
$\frac 1{\vert {\rm Box}(J_1)\vert} \frac 1{1-\ee^{-u_{i_1,J_1}}}$
and
$\frac 1{\vert {\rm Box}(J_2)\vert} \frac 1{1-\ee^{-u_{i_2,J_2}}}$, which give opposite residues
on $u_{i,J}=0$ in view of 
$$
\vert {\rm Box}(J_1)\vert u_{i_1,J_1}
= - {\vert {\rm Box}(J_2)\vert} u_{i_2,J_2}.
$$
Thus, the function used to define $\chi_\gamma(v)$ has no poles in codimension one and 
is holomorphic in the neighborhood of ${\bf u}=0$.

\medskip

As before, we see that $\chi_\gamma$ is well-defined on $K_0^c(\PP_\Sigma)$.
Observe also that $\chi_\gamma$ factors through $ch_\gamma:K_0^c(\PP_\Sigma)\to
H_\gamma^c$. To prove this, observe that for any $w=f(\log R)G_I$ we have 
$$
\chi_\gamma(w) = 
\Big(
\sum_{\stackrel{J\supseteq \sigma(\gamma)\cup I,}{ |J|=\rk N}}
\frac {f(2\pi\ii \gamma_i+\delta(i\in J)u_{i,J})}
{\vert {\rm Box}(J)\vert}
\prod_{i\in J-I} \frac 1{1-\ee^{-u_{i,J}-2\pi\ii\gamma_i}}\Big)_{{\bf u}=0},
$$
so one only needs to know the power series expansion of $f$ near $(2\pi\ii\gamma)$,
i.e. the information about the $\gamma$ local component of $K_0^c(\PP_\Sigma)$.

\medskip
Consider the fan $\Sigma_\gamma$ in the quotient
lattice $N_\gamma=N/(N\cap \QQ(\sigma(\gamma)))$ which comes from $J\supseteq \sigma(\gamma)$. Denote 
the corresponding cones and simplicial sets by $\bar J$. In the $\gamma$ component of $K_0^c(\PP_\Sigma)$ we have $R_i=1$ for $i\not\in {\rm Star}(\sigma(\gamma))$. We can also use the polynomial relations on $R_i$ to solve for $R_{i}$ in 
$\sigma(\gamma)$. We can thus consider only $v=(\prod_{i\in {\rm Star}(\sigma(\gamma))-\sigma(\gamma)}R_i^{\alpha_i})G_I$. In fact, we can use the invertibility of $(1-R_i^{-1})$ on the 
$\gamma$ component of $K_0^c(\PP_\Sigma)$ for $i\in \sigma(\gamma)$ to only consider 
$I$ with $I\cap \sigma(\gamma)=\emptyset$.
For each such $i\in {\rm Star}(\sigma(\gamma))-\sigma(\gamma)$ we will denote by $\bar i$ the corresponding
index in  the quotient fan and define $\alpha_{\bar i} = \alpha_i$.  We have
$$\vert {\rm Box}(J) \vert = \vert \Vol(\bar J)\vert\, \vert {\rm Box}(\sigma(\gamma))\vert,
$$ 
so using \eqref{this} and \eqref{chcgamma} we get 
$$
\chi_\gamma(v)=\int_\gamma f
$$
where $f$ is a piecewise linear function on $\Sigma_\gamma$ 
whose component $f_{\bar J}$ is
given by
$$\frac 1{ \vert {\rm Box}(\sigma(\gamma))\vert}
\prod_{i\in \sigma(\gamma)} \frac 1{(1-ch_\gamma(R_i^{-1}))}
\prod_{\bar i\in \bar J}\frac {u_{\bar i,\bar J}}{(1-\ee^{-u_{\bar i,\bar J}})}\,
ch^c_\gamma(\prod_{i}R_i^{\alpha_i}(G_I))_{\bar J}.
$$
The statement of proposition then follows.
\end{proof}

\begin{corollary}\label{perfect}
The pairing between $K_0(\PP_\Sigma)$ and $K_0^c(\PP_\Sigma)$ given 
by Definition \ref{Epairing} is non-degenerate.
\end{corollary}

\begin{proof}
Observe that 
$$
\chi(w,v) = \chi( w^\dual v)
$$
where $w\to w^\dual$ is the duality automorphism on $K_0(\PP_\Sigma)$ which
sends $R_i\to R_i^{-1}$.
The duality does not affect the nondegeneracy of the pairing. The rest follows from 
the fact that $(w,v)\mapsto \int_\gamma wv$ is a duality between $H_\gamma$ and $H_\gamma^c$,
and the 
fact that the corrections 
$$
\frac 1{\vert {\rm Box}(\sigma(\gamma))\vert}
\int_\gamma \prod_{i\in \sigma(\gamma)} \frac 1{(1-ch(R_i^{-1}))}
\prod_
{\stackrel {i\in{\rm Star}(\sigma(\gamma))}{i\not\in \sigma(\gamma)}}\Big(\frac{D_{\bar i}}{1-ch(R_i^{-1})}\Big)
$$
of Proposition \ref{hrr}
are invertible in $H_\gamma$.
\end{proof}

\begin{remark}
There are natural integer structures on $K_0(\PP_\Sigma)$ and $K_0^c(\PP_\Sigma)$ given
by looking at the integer linear combinations of monomials. It is clear from the proof of 
Proposition \ref{makessense} that the pairing on such monomials is integral. 
We do not know whether the pairing in question is unimodular.
\end{remark}

\begin{example}\label{pairingexample}
We now calculate in detail the case of $C\subset \ZZ^2$ with $v_1=(0,1)$, $v_2=(1,1)$ and $v_3=(3,1)$ considered in Example \ref{keyex}.

\medskip

We will  consider the natural bases of $K_0(\PP_\Sigma)$ and $K_0^c(\PP_\Sigma)$ given by $1,R_3, R_3^2$ and  $G_2,R_3G_2,R_3^2G_2.$
Our goal is to calculate the pairing between these basis elements. The key step is the following calculation of $\chi(R^kG_2)$.
Let us denote $q^{(a,b)}=s^at^b$. The Euler characteristic in question is the value at $s=t=1$ of
$$
\frac 1{1-s^{-1}t} +\frac 12\frac 1{1-s^{\frac 12}t^{-\frac 12}}s^{-\frac k2}t^{\frac k2} + \frac 12\frac 1{1+s^{\frac 12}t^{-\frac 12}}(-1)^ks^{-\frac k2}t^{\frac k2}.
$$
We can set $t=1$ to get 
$$
\frac 1{1-s}\Big(-s+\frac 12 s^{-\frac k2}(1+s^{\frac 12}) + \frac 12 (-1)^ks^{-\frac k2}(1-s^{\frac 12}) 
\Big)
$$
$$
=\frac 1{1-s}\Big(-s+\frac 12(1+(-1)^k)s^{-\frac k2} +\frac 12(1-(-1)^k)s^{-\frac k2 +\frac 12}
\Big).
$$
This means that 
$$
\chi(R_3^kG_2) = \left\{\begin{array}{ll}
\frac k2+1, &k=0\mod 2\\
\frac {k+1}2,&k=1\mod 2
\end{array}\right.
$$
and the pairings are given by the following table.
\begin{equation}\label{pairex}
\begin{array}{|r|ccc|}\hline
 & 1&R_3&R_3^2\\ 
\hline
G_2&1&0&0\\
R_3G_2&1&1&0\\
R_3^2G_2&2&1&1\\ \hline
\end{array}
\end{equation}
\end{example}

\begin{remark}
A motivated reader may want to check that Proposition \ref{hrr} holds in this example.
An unmotivated reader can rest assured that we checked it.
\end{remark}

\section{Derived categories}\label{sec:dcat}
The space  $K_0(\PP_\Sigma)$ is the (complexified) Grothendieck group of the triangulated category $D(\PP_\Sigma)$ of bounded complexes
of coherent sheaves on the stack $\PP_\Sigma$. 
In this section we will describe the triangulated category $D^c(\PP_\Sigma)$ 
whose Grothendieck group  is naturally isomorphic to $K_0^c(\PP_\Sigma)$.

\medskip

Note that the coarse moduli space of $\PP_\Sigma$ is not compact. However, it has a closed, proper over ${\rm Spec}(\CC)$, subscheme 
$\pi^{-1}(0)$ which is the union of the torus
orbits for cones $\sigma\in \Sigma$ such that the interior of $\sigma$ lies in 
the interior of $C$. Geometrically, this is the zero fiber 
of the map $\pi$ from 
$\PP_\Sigma$ to the affine toric variety 
${\rm Spec}\,\CC[C^\dual\cap N^\dual]$
defined by $C$, which motivates the notation.
The fiber $\pi^{-1}(0)$ need not be irreducible or even equidimensional. The following definition
is natural.
\begin{definition}
Denote by $D^c(\PP_\Sigma)$ the full subcategory of $D(\PP_\Sigma)$ which consists of objects whose cohomology sheaves are supported
on $\pi^{-1}(0)$. 
\end{definition}

\begin{theorem}\label{k0c}
The complexified Grothendieck group of $D^c(\PP_\Sigma)$ is isomorphic to $K^c_0(\PP_\Sigma)$.
\end{theorem}

\begin{proof}
We will denote the Grothendieck group of $D^c(\PP_\Sigma)$ by $G_0(D^c)$. This is a temporary notation, which will
be later replaced by $K^c_0(\PP_\Sigma)$ in view of this theorem.

\medskip

There is a map  $\mu:K^c_0(\PP_\Sigma)\to G_0(D^c)$ which sends
$\prod_i R_i^{k_i}G_I$ to the image of the pushforward of the appropriate line bundle on the closed substack of
$\PP_{\Sigma}$ which corresponds to $I$. The relations get sent to zero in view of Koszul complexes. We now want to show that this map $\mu$ is an isomorphism.

\medskip

First, we tackle the surjectivity of $\mu$. Every complex is equal in the Grothendieck group to the sum of its cohomology sheaves. 
Images of sheaves supported on $\pi^{-1}(0)$ in the Grothendieck group are generated by sheaves supported on irreducible
components of $\pi^{-1}(0)$. Moreover, one can filter by powers of the corresponding ideal, to reduce to pushforwards 
from the said irreducible components. Finally, sheaves on toric stacks are resolved by line bundles, see \cite{BH}.

\medskip

Injectivity requires a bit more work. We can use the pairing between $K_0^c(\PP_\Sigma)$ and $K_0(\PP_\Sigma)$ to  do this.
The key to this argument is the following lemma.

\begin{lemma}
There is a natural pairing between $D(\PP,\Sigma)$ and $D^c(\PP_\Sigma)$ given by 
$$
\langle F,G\rangle = \sum_i (-1)^k \dim_\CC {\rm Hom}(F,G[i]).
$$
This descends to the Grothendieck groups and coincides with the pairing defined in 
Corollary \ref{perfect}
on
the image of $K_0^c(\PP_\Sigma)$ defined above.
\end{lemma}

\begin{proof}
There is a spectral sequence that calculates ${\rm Hom}(F,G[i])$ starting from ${\rm Hom}$-s between $F$ and the shifts of the direct sum
of cohomology sheaves of $G$.  There is then a local-to-global spectral sequence that calculates 
these spaces in terms of cohomology of the local $\rm Ext$ sheaves, which are finite-dimensional because $\pi^{-1}(0)$ is proper.
Thus, the pairing is well-defined, and the passage to Grothendieck groups is then trivial.

\medskip

To verify that this pairing reduces to the pairing of Corollary \ref{perfect}, it is enough to consider 
the pairing between $\prod_i R_i^{\beta_i}$ and $\prod_i R_i^{\alpha_i}G_I$. It then follows 
from the proof of Proposition \ref{makessense}.
\end{proof}

We now complete the proof of Theorem \ref{k0c} by proving the injectivity of $\mu$. If an element of $K_0^c(\PP_\Sigma)$ is mapped to zero,
then its pairing with any element of $K_0(\PP_\Sigma)$ is zero. However, the pairing between $K_0(\PP_\Sigma)$ and $K_0^c(\PP_\Sigma)$
is perfect by Corollary \ref{perfect}, which finishes the proof.
\end{proof}

\section{Gamma series with values in $K_0(\PP_\Sigma)$ and $K_0^c(\PP_\Sigma)$}
\label{gamex}
In this section we describe the Gamma series solution to an appropriate version of 
GKZ hypergeometric system. Then we calculate these solutions explicitly 
for the Example \ref{keyex} in terms of Gamma function.
We also calculate the leading terms of some of these series, which will be useful in 
Section \ref{sec8}.

\medskip

Our combinatorial setup is as before. However, we now require that $v_i$ lie in
a hyperplane in $N$. We also assume that the fan $\Sigma$ is projective, in the sense that
it corresponds to a chamber of the secondary fan.
We recall the definitions of the better behaved hypergeometric system of equations 
from \cite{BH1}.

\begin{definition}\label{bbgkz}
Consider the system of partial differential equations on the infinite collection of 
functions $\Phi_c(x_1,\ldots,x_n)$ of $n$ variables. The functions are indexed by lattice elements $c\in C$. The differential equations are the following. For all $m\in M$, $c\in C$, $i\in \{1,\ldots,n\}$
\begin{equation}\label{gkz}
\partial_i \Phi_c = \Phi_{c+v_i},\hskip 30pt
\sum_{i=1}^n \langle m , v_i \rangle x_i\partial_i \Phi_{c} + \langle m,c\rangle \Phi_c =0.
\end{equation}
We call this system $bbGKZ(C,\beta=0)$ or $bbGKZ(C,0)$ for short. It was observed in \cite{BH1} that this system reduces to a system of holonomic PDEs
of a finite collection of functions.
Similarly, we define $bbGKZ(C^\circ,0)$ by considering $c\in C^\circ$ only.
\end{definition}

\begin{remark}
There is a restriction natural map from the space of solutions of 
$bbGKZ(C,0)$ to that of $bbGKZ(C^\circ,0)$. However, this map
is never an isomorphism, because it sends the trivial solution $\{\Phi_c = \delta_c^0\}$
to zero. The dimension of the image has been shown to be related to Erkhardt and
Stanley polynomials of $C$, see \cite{Bor.str}.
\end{remark}

\medskip

A modification of the original Gamma series of \cite{GKZ,stienstra,BH2} yields the 
following solution
of $bbGKZ(C,0)$ with values in $K_0^c$. We assume that we are given 
a projective
\footnote{Projectivity is necessary to assure uniform convergence in an open set of parameters.} simplicial subdivision $\Sigma$ of $C$. We also fix a choice of a branch
of $\log(x_i)$.

\begin{definition}\label{def:gamma}
Consider for each $c$ in $C$ and each twisted sector $\gamma = \sum_i \gamma_i v_i$
the set $L_{c,\gamma}$ of $(l_i)\in \QQ^n$ such that $\sum_i l_i v_ i = -c$ and $l_i-\gamma_i \in \ZZ$.
Define a collection of functions of $(x_1,\ldots,x_n)$ indexed by $c\in C$ with values in 
$\oplus_\gamma H_\gamma$
$$
(\Gamma(x_1,\ldots,x_n))_c = 
\bigoplus_\gamma\sum_{(l_i)\in L_{c,\gamma}}\prod_{i=1}^{n}
\frac {x_i^{l_i+\frac {D_i}{2\pi\ii}}}{\Gamma(1+l_i+\frac {D_i}{2\pi\ii})}
$$
where $x^a$ is defined by $\ee^{a\log x}$ after picking a branch of $\log x$.
Here $D_i=\log ch_\gamma(R_i\ee^{-2\pi\ii\gamma_i})$, in particular it is nilpotent.
Thus $x_i^\frac {D_i}{2\pi\ii}$ is well defined once the branch of $\log x$ is fixed.
\end{definition}

\begin{remark}
We are using the natural isomorphism $ch:K_0(\PP_\Sigma)\iso \oplus_\gamma H_\gamma$.
It allows us to view the above Gamma series as taking values in $K_0(\PP_\Sigma)$.
\end{remark}

\begin{proposition}\label{prop:gamma}
The Gamma series of Definition \ref{def:gamma} defines a solution of $bbGKZ(C,0)$
with values in $K_0(\PP_\Sigma)$. By composing with linear functions on $K_0(\PP_\Sigma)$ 
one gets all solutions of $bbGKZ(C,0)$.
\end{proposition}

\begin{proof}
Convergence is routine and is proved along the lines of \cite{BH2} where the 
case of $c=0$ is treated in detail. We leave this to the reader. The fact that this is a 
solution follows from $L_{c+v_i,\gamma}=L_{c,\gamma}-v_i$ and the defining property of 
the Gamma function. The second statement is a special case of \cite[Theorem 3]{H}.
\end{proof}

We now define an analogous Gamma series with values in $K_c^0(\PP_\Sigma)$.
Again, we fix the branches of $\log x_i$.

\begin{definition}\label{def:gammac}
For each $c\in C^\circ$, each twisted sector $\gamma$ and each element of $L_{c,\gamma}$ 
consider 
the set $\sigma$ of $i$ with $l_i\in \ZZ_{< 0}$.
If $\sigma\cup\sigma(\gamma)$ is not in a  cone of $\Sigma$, 
then we set the notation $F_\sigma$  to zero. 
We define
$$
(\Gamma^\circ(x_1,\ldots, x_n))_c=
\bigoplus_\gamma\sum_{(l_i)\in L_{c,\gamma}}\prod_{i=1}^{n}
\frac {x_i^{l_i+\frac {D_i}{2\pi\ii}}}{\Gamma(1+l_i+\frac {D_i}{2\pi\ii})}
(\prod_{i\in \sigma}D_i^{-1})
F_{\sigma}.
$$
with $D_i=\log ch_\gamma(R_i\ee^{-2\pi\ii\gamma_i})$.
\end{definition}

\begin{proposition}\label{prop:gammac}
The collection of series $(\Gamma^\circ)_c$ for $c\in C^\circ$ defines a solution of 
$bbGKZ(C^\circ,0)$ with values in $K^c_0(\PP_\Sigma)$. If one composes this with 
linearly independent linear functions on $K^c_0(\PP_\Sigma),$ one gets 
linearly independent solutions.
\end{proposition}

\begin{proof}
We have $\sum_i l_i v_i = - c$, so $\sigma$ consisting of
$i$ with $l_i\in \ZZ_{< 0}$
can not be a boundary simplex of $\Sigma_\gamma$. Indeed, 
since $c \in C^\circ,$ the simplex $\sigma$ is non--empty.
If $C_\gamma$ denotes the support of the fan $\Sigma_\gamma$
and $h:C_\gamma \to \RR_{\geq 0}$ 
is supporting function vanishing on $\sigma$, then 
$$
0>h(-c) = \sum_{i} l_i h(v_i) = \sum_{i\not\in \sigma(\gamma)}l_i h(v_i) =
\sum_{i\not\in\sigma\cup \sigma(\gamma)} l_ih(v_i)\geq 0.
$$
Thus $\Gamma^\circ(x)_c$ is well-defined as an element of $\oplus_\gamma H_\gamma^c$,
which is naturally isomorphic to $K_0^c(\PP_\Sigma)$ by Prop. \ref{prop:chc}.

\medskip

Convergence is again straightforward. It is also clear that this gives a solution
of $bbGKZ(C^\circ,0)$. For
the linear independence statement, it is clearly enough to
prove the statement for a fixed twisted sector.  
The statement is then a direct
consequence of \cite[Lemma 1 ii) and Prop. 4]{H} applied to the case 
$\beta =0.$ For, it is enough to note that, for fixed
$\lambda$ and $c,$ a term in the above summation defining the 
Gamma series $\Gamma^\circ$ is non-zero if and only if 
the corresponding term in the Gamma series summation for $\beta=0$
in loc. cit. is non--zero. In fact, the terms coincide up
to a rescaling of the classes $D_i$ by the $2 \pi i$ factor. 
The proof of \cite[Prop. 4]{H} applies then 
without any changes. 
\end{proof}

\begin{corollary} If $bbGKZ(C,0)$ and $bbGKZ(C^\circ,0)$ 
denote the spaces of solutions to the corresponding systems, then
the Gamma series functions $\Gamma : K_0 (\PP_\Sigma)^\vee \to bbGKZ(C,0)$
and $\Gamma^\circ : K_0^c (\PP_\Sigma)^\vee \to bbGKZ(C^\circ,0)$ 
are isomorphisms of linear spaces. 
\end{corollary}

\begin{proof} The first part simply restates 
the result of Proposition \ref{prop:gamma}. 
For the second part, note that for each $\gamma,$
$H^c_\gamma$ is the linear dual to $H_\gamma,$ by 
the results of Section \ref{sec:SR}. The isomorphisms
of Propositions \ref{prop:ch} and \ref{prop:chc} imply that
the complex vector spaces
$K_0 (\PP_\Sigma)$ and
$K_0^c (\PP_\Sigma)$ have the same dimension, and 
the result follows from the previous 
proposition.
\end{proof}

\begin{example}
We will now calculate the Gamma series in the case of  $C\subset \ZZ^2$ with 
$v_1=(0,1)$, $v_2=(1,1)$ and $v_3=(3,1)$ that we considered in Example \ref{keyex}.
We first calculate $\Gamma(x_1,x_2,x_3)$ with values in $\oplus_\gamma H_\gamma=
K_0(\PP_\Sigma)$.
The solution is uniquely determined by $(\Gamma(x_1,x_2,x_3))_{(0,0)}$ and 
$(\Gamma(x_1,x_2,x_3))_{(2,1)}$, because all other components of the solution
are partial derivatives of one or both of these two functions.

\medskip\noindent
We have $\gamma_1=(0,0)$ and $\gamma_2=(2,1)$. The sets $L_{c,\gamma}$ are 
summarized in the following table in terms of $(l_1,l_2,l_3)$. We also include in the 
table the case of $c=(1,1)$ which is used in the later calculation of the Gamma series
solution of $bbGKZ(C^\circ,0)$.
$$
\begin{array}{|l|l|}\hline
c,\gamma& L_{c,\gamma}\\ \hline
(0,0),(0,0) & \ZZ(2,-3,1)\\
(0,0),(2,1) & (-1,\frac 32,-\frac 12)+\ZZ(2,-3,1) \\
(2,1),(0,0) & (-1,1,-1)+\ZZ(2,-3,1) \\
(2,1),(2,1) & (0,-\frac 12,-\frac 12)+\ZZ(2,-3,1)\\ 
(1,1),(0,0) & (0,-1,0)+\ZZ(2,3,-1)\\
(1,1),(2,1) & (-1,\frac 12,-\frac 12)+\ZZ(2,3,-1)\\[.2em]
\hline
\end{array}
$$
We write the corresponding summations. To shorten the notation, we use 
$$
x=x_1^2x_2^{-3}x_3.
$$
For the twisted sector $\gamma=(2,1)$ 
the $D_i$ are $0$, therefore
$$
\begin{array}{ll}
\Gamma(x_1,x_2,x_3)_{(0,0)}
&=
\sum_{k} 
\frac{x^k x_1^{\frac {D_1}{2\pi\ii}}x_2^{\frac {D_2}{2\pi\ii}}x_3^{ \frac {D_3}{2\pi\ii}}}
{\Gamma(1+2k+\frac {D_1}{2\pi\ii})
\Gamma(1-3k+\frac {D_2}{2\pi\ii})
\Gamma(1+k+\frac { D_3}{2\pi\ii})}
\\
&
\oplus (x_1^{-1}x_2^{\frac 32}x_3^{-\frac 12})\sum_{k}\frac {x^k}
{\Gamma(2k)
\Gamma(\frac 52-3k)
\Gamma(\frac 12+k)}
\\
&
=
\sum_{k\geq 0} 
\frac{x^k (1+\frac {D_3}{2\pi\ii}\log x)
}
{\Gamma(1+2k+\frac {2D_3}{2\pi\ii})
\Gamma(1-3k-\frac {3D_3}{2\pi\ii})
\Gamma(1+k+\frac {D_3}{2\pi\ii})}
\\
&
\oplus (x_1^{-1}x_2^{\frac 32}x_3^{-\frac 12})\sum_{k\geq 1}\frac {x^k}
{(2k-1)!
\Gamma(\frac 52-3k)
\Gamma(\frac 12+k)}
\\
&
=1+ \frac {D_3}{2\pi\ii}\log x+\frac {3D_3}{2\pi \ii}
\sum_{k\geq 1} 
\frac{x^k  (3k-1)!(-1)^k}
{(2k)!k!}
\\
&
\oplus (x_1^{-1}x_2^{\frac 32}x_3^{-\frac 12})\sum_{k\geq 1}\frac {x^k}
{(2k-1)!
\Gamma(\frac 52-3k) \Gamma(\frac 12+k)}\\
&= 1+ \frac {D_3}{2\pi\ii} (\log x - 3x+\frac {15} 2  x^2+\ldots)
\\
&
\oplus (x_1^{-1}x_2^{\frac 32}x_3^{-\frac 12})
\frac 1\pi (-x + \frac {35}{24}x^2- \frac {3003}{640}x^3+\ldots).
\end{array}
$$
A similar calculation which we omit shows that 
$$
\begin{array}{ll}
\Gamma(x_1,x_2,x_3)_{(2,1)}
&=
\frac {D_3}{2\pi\ii}(x_1^{-1}x_2x_3^{-1})
(3x-12x^2+63x^3+\ldots)
\\
&
\oplus (x_2^{-\frac 12}x_3^{-\frac 12})
\frac 1\pi(1-\frac {15}8x +  \frac{1155}{128}x^2+\ldots).
\end{array}
$$

\medskip\noindent
We will now write down the Gamma series solution $\Gamma^\circ$  of $bbGKZ(C^\circ, 0)$ with values in $K_0^c(\PP_\Sigma)$. It is determined uniquely by 
$\Gamma^\circ(x_1,x_2,x_3)_{(1,1)}$ and $\Gamma^\circ(x_1,x_2,x_3)_{(2,1)}$.
This calculation is more delicate, as we need to use the values of the derivative
of the Gamma function. We use the standard results about the special values of 
the logarithmic derivative $\psi$ of the Gamma function to find the following.
$$
\begin{array}{ll}
&\Gamma^{\circ}_{(1,1)} 
= x_2^{-1}\sum_{k} \frac{x^k(1+\frac {D_3}{2\pi\ii}\log x)}
{
\Gamma(1+2k+\frac {D_1}{2\pi\ii})
\Gamma(-3k+\frac {D_2}{2\pi\ii})
\Gamma(1+k+\frac {D_3}{2\pi\ii})
}D_2^{-1}F_2
\\
&
\oplus x_1^{-1}x_2^{\frac 12}x_3^{-\frac 12}
\sum_{k} 
\frac 
{x^k
}
{\Gamma(2k)\Gamma(\frac 32-3k)\Gamma(\frac 12+k)}
\bar F_\emptyset
\\
&
= x_2^{-1}\sum_{k\geq 0} \frac{x^k(1+\frac {D_3}{2\pi\ii}\log x)
\Gamma(3k+1+3\frac {D_3}{2\pi\ii})(-1)^{k}
}
{
\Gamma(1+2k+\frac {2D_3}{2\pi\ii})
\Gamma(1+k+\frac {D_3}{2\pi\ii})
}\frac{1}{2\pi\ii}
F_2
\\
&
\oplus x_1^{-1}x_2^{\frac 12}x_3^{-\frac 12}
\sum_{k\geq 1} 
\frac 
{x^k
}
{\Gamma(2k)\Gamma(\frac 32-3k)\Gamma(\frac 12+k)}
\bar F_\emptyset
\\
&
= x_2^{-1}\sum_{k\geq 0} \frac{x^k(-1)^{k}(3k)!}
{
(2k)!k!}\frac{1}{2\pi\ii}
F_2
+
x_2^{-1}\sum_{k\geq 0} 
\frac{x^k(-1)^{k}(3k)!}{(2k)!k!}
\\&
\cdot(\log x+3\psi(3k+1)-2\psi(2k+1)-\psi(k+1))
\frac {D_3}{2\pi\ii}F_2
\\
&
\oplus x_1^{-1}x_2^{\frac 12}x_3^{-\frac 12}
\sum_{k\geq 1} 
\frac 
{x^k
}
{\Gamma(2k)\Gamma(\frac 32-3k)\Gamma(\frac 12+k)}
\bar F_\emptyset
\\&
=x_2^{-1}(1-3x+15x^2+\ldots)\frac{1}{2\pi\ii}
F_2
\\&
+x_2^{-1}(\log x(1-3x+15x^2+\ldots)-\frac 92 x+\frac {101}4x^2+\ldots)\frac {D_3}{2\pi\ii}\frac{1}{2\pi\ii}
F_2
\\
&
\oplus
 x_1^{-1}x_2^{\frac 12}x_3^{-\frac 12}
( \frac 32 x-\frac {105}{16}x^2
+\frac {9009}{256}x^3
+\ldots 
)
\frac 1\pi
\bar F_\emptyset
\end{array}
$$
Similarly,
$$
\begin{array}{ll}

&\Gamma^{\circ}_{(2,1)}
=
x_1^{-1}x_2x_3^{-1}\sum_{k\geq 1} 
\frac{x^k(1+\frac {D_3}{2\pi\ii}\log x)}{\Gamma(2k+\frac {D_1}{2\pi\ii})
\Gamma(2-3k+\frac {D_2}{2\pi\ii})
\Gamma(k+\frac { D_3}{2\pi\ii})}
D_2^{-1}F_2
\\
&
\oplus x_2^{-\frac 12}x_3^{-\frac 12}\sum_{k\geq 0}\frac {x^k}
{\Gamma(1+2k)
\Gamma(\frac 12-3k)
\Gamma(\frac 12+k)}\bar F_\emptyset
\\
&
=x_1^{-1}x_2x_3^{-1}\sum_{k\geq 1} 
\frac{x^k(1+\frac {D_3}{2\pi\ii}\log x)
\Gamma(3k-1+\frac {3D_3}{2\pi\ii})(-1)^k
}
{\Gamma(2k+\frac {D_1}{2\pi\ii})
\Gamma(k+\frac { D_3}{2\pi\ii})}
\frac 1{2\pi\ii}F_2
\\&
\oplus x_2^{-\frac 12}x_3^{-\frac 12}\sum_{k\geq 0}\frac {x^k}
{\Gamma(1+2k)
\Gamma(\frac 12-3k)
\Gamma(\frac 12+k)} \bar F_\emptyset
\\
&
=x_1^{-1}x_2x_3^{-1}\sum_{k\geq 1} 
\frac{x^k
\Gamma(3k-1)(-1)^k}
{\Gamma(2k)\Gamma(k)}
\frac 1{2\pi\ii}F_2 \\
&+x_1^{-1}x_2x_3^{-1}
 \sum_{k\geq 1} 
\frac{x^k
\Gamma(3k-1)(-1)^k}
{\Gamma(2k)\Gamma(k)} \\
&\cdot
(\log x +3\psi(3k-1)-2\psi(2k)-\psi(k))
\frac {D_3}{2\pi\ii}\frac 1{2\pi\ii}F_2
\\&
\oplus \frac 1{\pi}x_2^{-\frac 12}x_3^{-\frac 12}
(1-\frac {15}8 x +\frac {1155}{128}x^2+\ldots) \bar F_\emptyset
\\
&
=x_1^{-1}x_2x_3^{-1}
(-x+4x^2-21x^3+\ldots)
\frac 1{2\pi\ii}F_2
\\
&
+x_1^{-1}x_2x_3^{-1}(\log x(-x+4x^2-21x^3+\ldots)-x+\frac {19}3x^2-\ldots)
\frac {D_3}{2\pi\ii}\frac 1{2\pi\ii}F_2
\\&
\oplus x_2^{-\frac 12}x_3^{-\frac 12}
(1-\frac {15}8 x +\frac {1155}{128}x^2+\ldots) 
\frac 1\pi\bar F_\emptyset.
\end{array}
$$
\end{example}

\begin{remark}
The solutions to $bbGKZ(C,0)$  can be written
in terms of elementary functions.
We will spare the reader the lengthy calculation
that we performed using various symbolic manipulation software packages. Our
method was to rewrite the PDE as third order ordinary differential equations with 
respect to $x$ and then solve them symbolically using standard software. The solutions 
are then matched to the Gamma series using the asymptotic expansions at $x=0$. The final
answers are that $\Gamma_{(0,0)}$ can be written as the function
$$
\begin{array}{l}
1 + \frac{D_3}{2\pi\ii}  
\Big(3 \log\Big(( \sqrt x  + \sqrt{ 4/27 + x} )^{\frac 13}
 -
(- \sqrt x  + \sqrt{ 4/27 + x} )^{\frac 13}\Big) \\
   - \frac 12\log (4x) \Big) 
\oplus
(-\frac 2\pi)  
  \arctan \Big(
  \frac {-2^{\frac 23} + 3 (\sqrt x + \sqrt{4/27 + x})^\frac 23}
  {\sqrt 3 (2^{\frac 23} + 3  (\sqrt x + \sqrt{4/27 + x})^\frac 23)}
  \Big),
\end{array}
$$
and $\Gamma_{(2,1)}$ as 
$$
\begin{array}{l}
\frac {D_3}{2\pi\ii}
x_1^{-1}x_2x_3^{-1} 
\frac{9\cdot 2^{\frac 43}x}
{
2^{4/3} +
(-\sqrt {27 x} + \sqrt{4+27x} )^{\frac 43} +
(\sqrt {27 x} + \sqrt{4+27x} )^{\frac 43} 
}\\
\oplus
\frac {x_2^{-\frac 12}x_3^{-\frac 12}}{2^{\frac 23}\pi}
\Big(\frac{(-\sqrt{27 x} + \sqrt{4+27 x})^{\frac 23} + (\sqrt{27 x} + \sqrt{4+27 x})^{\frac 23})
}
{ \sqrt {4+27 x}}\Big).
\end{array}
$$
We were unable to find an analogous description of $\Gamma^\circ$. It appears that the 
coefficient of $D_3F_2$ is not an elementary function.
\end{remark}

\section{The conjectures}\label{sec:conj}
We are now ready to formulate several conjectures regarding the Gamma series 
constructions of this paper. In the subsequent sections we will explain the evidence that 
supports them.

\medskip

We expect the  Gamma series solutions of Definitions \ref{def:gamma} and 
\ref{def:gammac} to be compatible 
with pullback-pushforward and analytic continuation. Specifically, let
$\Sigma_1$ and $\Sigma_2$ be two adjacent triangulations, so that 
the stacks $\PP_{\Sigma_1}$ and $\PP_{\Sigma_2}$ differ by a 
flop that's a composition of weighted blowup and weighted blowdown.
There is a natural pullback-pushforward functor from the derived category 
of coherent sheaves on $\PP_{\Sigma_1}$ to that of $\PP_{\Sigma_2}$.
This functor preserves the property of support of cohomology supported on
a compact toric substack.
Indeed, the maps are compatible with the restriction to the 
preimage of the complement of the origin. As a consequence
of \cite[Theorem 4.2]{Kaw1}, the 
pullback-pushforward functors $D(\PP_{\Sigma_1}) \to 
D(\PP_{\Sigma_2})$ and $D^c(\PP_{\Sigma_1}) \to 
D^c(\PP_{\Sigma_2})$ are equivalences. 
They induce natural group isomorphisms
$pp:K_0(\PP_{\Sigma_1})\to K_0(\PP_{\Sigma_2})$ and 
$pp:K_0^c(\PP_{\Sigma_1})\to K_0^c(\PP_{\Sigma_2})$.

\medskip
\begin{conjecture}\label{conj.ac}
The following diagrams of isomorphisms are commutative
$$
\begin{array}{ccc}
K_0(\PP_{\Sigma_2})^\dual&\stackrel{pp^\dual}\longrightarrow& K_0(\PP_{\Sigma_1})^\dual
\\[.3em]
\Gamma\downarrow~{~}&                             &~{~} \downarrow\Gamma
\\[.3em]
bbGKZ(C,0) &\stackrel{a.c.}\longrightarrow&bbGKZ(C,0)
\end{array}
$$
$$
\begin{array}{ccc}
K^c_0(\PP_{\Sigma_2})^\dual&\stackrel{pp^\dual}\longrightarrow& K^c_0(\PP_{\Sigma_1})^\dual
\\[.3em]
\Gamma^\circ\downarrow~{~}&                             & ~{~}\downarrow\Gamma^\circ
\\[.3em]
bbGKZ(C^\circ,0) &\stackrel{a.c.}\longrightarrow&bbGKZ(C^\circ,0)
\end{array}
$$
where the top rows are the duals of the maps induced by the pullback-pushforward
derived functors,
and the bottom rows are analytic continuations along  a certain path in the 
domain of parameters considered in \cite{BH2}.
\end{conjecture}

\begin{remark}
The motivation behind this conjecture comes from Homological Mirror Symmetry.
A particular case of this phenomenon was first observed in the second author's thesis.
It has been extended to the case of stacks in \cite{BH2}. However, in that 
paper we were using the original GKZ system, rather than a better-behaved one,
so the vertical maps were not isomorphisms in general.
The advantage of the better-behaved version is that all of the maps in the above conjecture
are isomorphisms.
It appears plausible that  methods similar to those of \cite{BH2} will suffice in
our case. However, the details are far from settled at this time. 
\end{remark}

\medskip

It  seems plausible from the approach of 
\cite{MMW,Walther1} that the systems of PDEs 
$bbGKZ(C,0)$ 
and $bbGKZ(C^\circ,0)$ are dual to each other. We formulate here a more specific
conjectural mechanism of such duality.
\begin{conjecture}\label{conj.pairing}
There exists a collection $p_{c,d}(x_1,\ldots,x_n)$ of polynomials
indexed by 
$c\in C$, $d\in C^\circ,$ such that the following hold.
\begin{itemize}
\item All but a finite number of $p_{c,d}$ are zero.

\item For any pair of solutions $(\Phi_c)$ and $(\Psi_d)$ of $bbGKZ(C,0)$ and
$bbGKZ(C^\circ,0)$ respectively the sum
\begin{equation}\label{eq:pairing}
\sum_{c,d} p_{c,d}\Phi_c\Psi_d
\end{equation}
is constant as a function of $(x_1,\ldots,x_n)$.

\item
The pairing provided by (\ref{eq:pairing}) is non-degenerate. 

\item
For any projective simplicial subdivision $\Sigma$, 
the pairing provided by (\ref{eq:pairing}) is the inverse of the 
Euler characteristics pairing between $K_0(\PP_\Sigma)$ and 
$K^c_0(\PP_\Sigma)$ under the $\Gamma$ and $\Gamma^\circ$ series.
\end{itemize}
\end{conjecture}

\begin{remark}
The polynomials $p_{c,d}$ are not unique in view of the relations between 
the components of $\Phi$ and $\Psi$. However, one may hope for a
somewhat special choice of $p$.
\end{remark}

\begin{remark}
Conjectures \ref{conj.ac} and \ref{conj.pairing} can likely be categorified
to produce two families of triangulated categories over the parameter space
of $x$ such that in the neighborhoods of toric degeneracy points the Gamma series
provide some kind of character map to the corresponding $K$-theory. However,
such categories have not yet been constructed.
\end{remark}

\section{Example: The pairing}\label{sec8}
Consider the spaces of solutions $(\Phi)_c$ and $(\Psi)_d$ for $c\in C$, $d\in C^\circ$
of $bbGKZ(C,0)$ and $bbGKZ(C^\circ,0)$ respectively for the cone spanned
by $v_1=(0,1)$, $v_2=(1,1)$, $v_3=(3,1)$. Recall that these are collections of functions of $x_1,x_2,x_3$.

\begin{proposition}\label{explicit}
For each $\Phi$ and $\Psi$ define
$$
\langle \Phi,\Psi\rangle =-2x_2x_3\Phi_{(2,1)}\Psi_{(2,1)}+
\sum_{\stackrel{c\in C,d\in C^\circ}{ c+d = v_1+v_2}} x_1x_2 \Phi_c\Psi_d (-1)^{\deg c} 
$$
$$+\sum_{\stackrel{c\in C,d\in C^\circ}{ c+d = v_1+v_3}}9x_1x_3 \Phi_c\Psi_d (-1)^{\deg c} 
+\sum_{\stackrel{c\in C,d\in C^\circ}{ c+d = v_2+v_3}} 4x_2x_3 \Phi_c\Psi_d (-1)^{\deg c}
$$
where $\deg (a,b)=b$. 
Then this function of $(x_1,x_2,x_3)$ is constant. Moreover, the resulting pairing of the 
solution spaces is perfect.
\end{proposition}

\begin{proof}
The 
linear relations on $\Phi_c$ and  $\Psi_d$ imply that $\langle \Phi,\Psi\rangle$
depends on $x_1^2x_2^{-3}x_3$ only.  So to prove that it is constant we only need
to verify $\frac\partial{\partial x_2}\langle \Phi,\Psi\rangle=0$. This is a consequence of 
the relations on $\Phi$ and $\Psi$ as follows. We have
$$
\begin{array}{rl}
\langle \Phi,\Psi\rangle
=&x_1x_2(\Phi_{(0,0)}\Psi_{(1,2)}-\Phi_{(0,1)}\Psi_{(1,1)})
\\
+&9x_1x_3(\Phi_{(0,0)}\Psi_{(3,2)}-\Phi_{(1,1)}\Psi_{(2,1)}-\Phi_{(2,1)}\Psi_{(1,1)})
\\
+&x_2x_3(4\Phi_{(0,0)}\Psi_{(4,2)}-6\Phi_{(2,1)}\Psi_{(2,1)}-4\Phi_{(3,1)}\Psi_{(1,1)}),
\end{array}
$$
so
$$
\begin{array}{rl}
\frac\partial{\partial x_2} & \langle \Phi,\Psi\rangle
=
x_1(\Phi_{(0,0)}\Psi_{(1,2)}-\Phi_{(0,1)}\Psi_{(1,1)})
\\
+&x_3(4\Phi_{(0,0)}\Psi_{(4,2)}-6\Phi_{(2,1)}\Psi_{(2,1)}-4\Phi_{(3,1)}\Psi_{(1,1)})
\\
+&x_1x_2(\Phi_{(1,1)}\Psi_{(1,2)}+\Phi_{(0,0)}\Psi_{(2,3)}-\Phi_{(1,2)}\Psi_{(1,1)}
-\Phi_{(0,1)}\Psi_{(2,2)})
\\
+
&9x_1x_3(\Phi_{(0,0)}\Psi_{(4,3)}-
\Phi_{(2,2)}\Psi_{(2,1)}-
\Phi_{(3,2)}\Psi_{(1,1)}-
\Phi_{(2,1)}\Psi_{(2,2)})
\\
+
&
x_2x_3(4\Phi_{(1,1)}\Psi_{(4,2)}+
4\Phi_{(0,0)}\Psi_{(5,3)}-
6\Phi_{(3,2)}\Psi_{(2,1)}
\\
-
&6\Phi_{(2,1)}\Psi_{(3,2)}-
4\Phi_{(4,2)}\Psi_{(1,1)}-
4\Phi_{(3,1)}\Psi_{(2,2)}).
\end{array}
$$
We use equations 
$$
x_2\Phi_{(a+1,b+1)}+3x_3\Phi_{(a+3,b+1)}+a\Phi_{(a,b)}=0
$$
and similarly for $\Psi$ to get rid of $x_1x_2$ terms above.
We use equations 
$$
3x_1\Phi_{(a,b+1)}+2x_2\Phi_{(a+1,b+1)}+(3b-a)\Phi_{(a,b)}=0
$$
and similarly for $\Psi$ to get rid of $x_2x_3$ terms.
After cancellations, we get
$$
\frac\partial{\partial x_2}\langle \Phi,\Psi\rangle
=3x_1x_3(\Phi_{(3,1)}\Psi_{(1,2)} - \Phi_{(0,1)}\Psi_{(4,2)})
$$
which equals zero in view of relations
$$
x_1\Phi_{(0,1)}-2x_3\Phi_{(3,1)} = x_1\Psi_{(1,2)}-2x_3\Psi_{(4,2)}=0.
$$
We have thus proved that 
$\langle \Phi,\Psi\rangle$ is constant.

\medskip

To verify that the pairing is perfect, we will analyze its results on the 
Gamma series solutions from Section \ref{gamex}. It is enough to calculate
the leading terms of the expressions. 
We use $e_{(0,0)}$ and $e_{(2,1)}$ as notation
to designate the identity parts in different twisted sectors for the $\Gamma$ series.
This calculation shows
\begin{equation}\label{ggc}
\langle \Gamma, \Gamma^\circ \rangle
=-\frac 3{2\pi^2} (e_{(0,0)} \otimes D_3F_2 + 4e_{(2,1)} \otimes \bar F_\emptyset
- D_3e_{(0,0)}\otimes F_2)
\end{equation}
in the tensor product $(\oplus_\gamma H_\gamma) \otimes (\oplus_\gamma H_\gamma^c)$.
\end{proof}

We now want to verify Conjecture \ref{conj.pairing} in our case.
Recall that $ch=ch_{(0,0)}\oplus ch_{(2,1)}$ (resp. 
$ch^c=ch^c_{(0,0)}\oplus ch^c_{(2,1)}$) identifies 
$K_0(\PP_\Sigma)$ (resp. $K^c_0(\PP_\Sigma)$) with
$H_{(0,0)}\oplus H_{(2,1)}$ (resp. $H^c_{(0,0)}\oplus H^c_{(2,1)}$).
They have been calculated in \eqref{chmaps}.
We use their inverses to rewrite \eqref{ggc} as
$$
\begin{array}{rl}
\langle \Gamma, \Gamma^\circ \rangle
=&-\frac 3{2\pi^2}(\frac 14 (3+2R_3-R_3^2) \otimes \frac 12(R_3^2G_2-G_2) \\
&+ 
4\frac 14(1-2R_3+R_3^2)
\otimes \frac 18 (G_2-2R_3G_2+R_3^2G_2)
\\
&
- \frac 12(R_3^2-1) \otimes \frac 12(3G_2+ R_3G_2 -2R_3^2G_2))
\\
=&-\frac 3{16\pi^2}((3+2R_3-R_3^2) \otimes (R_3^2G_2-G_2) \\
&+ (1-2R_3+R_3^2)
\otimes (G_2-2R_3G_2+R_3^2G_2)
\\
&
- 2(R_3^2-1) \otimes (3G_2+ R_3G_2 -2R_3^2G_2))
\\
=&-\frac 3{4\pi^2}(1\otimes G_2 -R_3\otimes G_2 - R_3^2\otimes R_3 G_2
+ R_3\otimes R_3G_2
\\
&
-R_3^2\otimes G_2+R_3^2\otimes R_3^2G_2)
\end{array}
$$
It remains to observe that this is precisely $-\frac 3{4\pi^2}$ times the inverse of the pairing
on $K_0(\PP_\Sigma)$ and $K_0^c(\PP_\Sigma)$ calculated in \eqref{pairex}. Thus 
we have found  a pairing that satisfies Conjecture  \ref{conj.pairing}.

\medskip

While we have verified the conjecture for $\Sigma$ that uses $v_2$, we also need to 
check that \emph{the same pairing} works for the only other possible fan, namely the one 
that does not use $v_2$. In this case, there are three twisted sectors, indexed by
$(0,0)$, $(1,1)$ and $(2,1)$.  Each of the corresponding spaces $H_\gamma$ and 
$H_\gamma^c$ is one-dimensional. We will briefly go through the key steps of the 
calculation while leaving the details to the reader. We recycle the notation from
the other examples, but they are now applied to a different situation. We use the notation
$w=\ee^{\frac {2\pi\ii}3}$.
$$\begin{array}{l}
K_0(\PP_\Sigma) = \CC[R_3]/\langle 1-R_3^3\rangle,~~
K^c_0(\PP_\Sigma) = \CC[R_3]G_{13}/\langle (1-R_3^3)G_{13}\rangle
\\
\chi(R_3^k G_{13}) = \delta_{k}^{0\hskip-7pt\mod 3}
\\
ch(a_1+a_2R_3+a_3R_3^2) = (a_1+a_2+a_3)e_{(0,0)} \\
\hskip 10pt
\oplus (a_1+wa_2+w^2a_3)e_{(1,1)}
\oplus (a_1+w^2a_2+wa_3)e_{(2,1)}\\
ch^c( G_{13}) = F_{13}\oplus 3\bar F_{\emptyset,(1,1)}\oplus 3\bar F_{\emptyset,(2,1)}.
\end{array}
$$
The Gamma series are expanded in terms of the powers 
of $x^{-1}$. 
$$\begin{array}{l}
\Gamma_{(0,0)} = e_{(0,0)}\oplus x_1^{-\frac 43}x_2^{2}x_3^{-\frac 23} (
-\frac {\sqrt 3}{12\pi}+\frac {7\sqrt 3}{2430\pi}x^{-1}+\ldots)e_{(1,1)}
\\
\hskip 50pt
\oplus x_1^{-\frac 23}x_2 x_3^{-\frac 13}(\frac {\sqrt 3}{2\pi}-\frac {5\sqrt 3}{648\pi}x^{-1}+\ldots)e_{(2,1)}
\\
\Gamma_{(2,1)} = 0e_{(0,0)}
\oplus x_1^{-\frac 13} x_3^{-\frac 23} (
\frac {\sqrt 3}{2\pi}-\frac {2\sqrt 3}{81\pi}x^{-1}+\ldots)e_{(1,1)}
\\
\hskip 50pt
\oplus x_1^{-\frac 53}x_2^2 x_3^{-\frac 43} (\frac {\sqrt 3}{18\pi}
-\frac {4\sqrt 3}{729\pi}x^{-1} + \ldots)e_{(2,1)}
\\
\Gamma_{(1,1)}^\circ = x_1^{-2}x_2^{2}x_3^{-1}(\frac 1{8\pi^2}-\frac 1{80\pi^2}x^{-1}+\ldots )F_{13}
\\
\hskip 50pt \oplus x_1^{-\frac 43}x_2x_3^{-\frac 23} (
-\frac {\sqrt 3}{6\pi}+\frac {7\sqrt 3}{486\pi}x^{-1}+\ldots)^{-1} \bar F_{\emptyset,(1,1)}
\\
\hskip 50pt
\oplus x_1^{-\frac 23} x_3^{-\frac 13}(\frac {\sqrt 3}{2\pi}-\frac {5\sqrt 3}{162\pi}x^{-1}+\ldots)\bar F_{\emptyset,(2,1)}
\\
\Gamma_{(2,1)}^\circ = x_1^{-1}x_2x_3^{-1}(-\frac 1{4\pi^2}+\frac 1{48\pi^2}x^{-1}+\ldots)F_{13}
\\
\hskip 50pt
\oplus x_1^{-\frac 13} x_3^{-\frac 23} (
\frac {\sqrt 3}{2\pi}-\frac {2\sqrt 3}{81\pi}x^{-1}+\ldots)^{-1} \bar F_{\emptyset,(1,1)}
\\
\hskip 50pt
\oplus x_1^{-\frac 53}x_2^2 x_3^{-\frac 43} (\frac {\sqrt 3}{18\pi}-\frac {4\sqrt 3}{729\pi}x^{-1} +\ldots)\bar F_{\emptyset,(2,1)}.\\
\end{array}
$$
The pairing gives 
$$
\langle \Gamma, \Gamma^\circ \rangle=-\frac 9{4\pi^2}(3e_{(2,1)}\otimes
\bar F_{\emptyset,(1,1)} + e_{(0,0)}\otimes F_{13} +3 e_{(1,1)}\bar F_{\emptyset,(2,1)})
$$
$$
=-\frac 3{4\pi^2}(1\otimes G_{13}+R_3\otimes R_3G_{13} + R_3^2\otimes R_3^2G_{13})
$$
as predicted.

\begin{proposition}
Conjectures \ref{conj.ac} and \ref{conj.pairing} hold in the Example $v_1=(0,1)$, $v_2=(1,1)$, $v_3=(3,1)$.
\end{proposition}

\begin{proof}
We have verified Conjecture \ref{conj.pairing} above. Observe that the solutions of 
$bbGKZ(C,0)$ are uniquely determined by $(0,0)$ term. As a result, the system is equivalent 
to the usual GKZ system. The analytic continuation statement is already proved in
\cite{BH2}, thus it holds for our $\Gamma$ series. Because both Euler characteristics 
and the pairing of Conjecture \ref{conj.pairing} are unchanged under the pullback-pushforward
and analytic continuation respectively, the statement for $\Gamma^\circ$ follows.
\end{proof}

\section{Pairing with $1$}
In this section we provide important evidence in favor of Conjecture \ref{conj.pairing}
which gives hints at a possible explicit description of the pairing. Specifically,
we can be reasonably certain of the part of the pairing in Conjecture \ref{conj.pairing}
that involves $\Phi_0$ only. In terms of $K$-theory this corresponds to a
well-defined linear function on $K_0(\PP_\Sigma)$, namely the rank at the generic point
as we see in the following proposition.

\medskip
\begin{proposition}\label{delta}
Consider the linear function 
$$
\rk:K_0(\PP_\Sigma)\to \CC
$$
defined by $\rk (\prod_i R_i^{l_i})=1$ for any collection $(l_1,\ldots,l_n)\in \ZZ^n$.
Then $\rk(\Gamma(x_1,\ldots,x_n))_c = \delta_c^0$.
\end{proposition}

\begin{proof}
First, we observe that $\rk$ is well-defined. Indeed, it clearly vanishes on 
all relations \eqref{Kthry}. Second, under the isomorphism of $K_0(\PP_\Sigma)=
\oplus H_\gamma$, the corresponding linear functions on the $\gamma\neq 0$ 
sectors are zero, because at least one $R_i-1$ is invertible there.  Similarly,
for the $H_0$ sector the function $\rk$ reads off the degree zero component only.

\medskip

Since the function $\rk$ vanishes on $H_{\gamma}$, $\gamma\not= 0,$
we only need to see what happens to the terms 
$$
\rk \Big(\prod_i \frac {x_i^{l_i+\frac {D_i}{2\pi\ii}}}{\Gamma (1+l_i+\frac{D_i}{2\pi\ii})}\Big)
$$
with $(l_i)\in L_{c,0}$. Since $\rk$ vanishes on all positive degree monomials
in $D_i$, we need $l_i\geq 0$. Since $\sum_i l_i v_i = -c$, the only nonzero
terms occur for
$l_i=0$ for all $i$ and $c=0$. In that case, we have $\rk(1)=1$.
\end{proof}

\medskip\noindent
If we consider the pairing of  Conjecture \ref{conj.pairing} applied to
$\Phi_c=\delta_c^0$ and arbitrary $\Psi$, then this pairing only 
involves the $p_{0,d}$. In what follows we show that there is a natural 
choice of such $p_{0,d}$ so that $\sum_{d\in C^\circ} p_{0,d}\Psi_d$ is 
constant and is compatible with Gamma series and the pairing on K-theories.

\medskip

Recall, see for example \cite[Def. 2.3]{batyrevmaterov}, that the
logarithmic Hessian of $\sum_i x_i[v_i]$
is an element of the semigroup ring $\CC[C]$ given by 
$$
{\rm Hessian}(x_1,\ldots,x_n)=\sum_{\vert I\vert =\rk N}\Vol_I^2( \prod_{i\in I}x_i )[\sum_{i\in I}v_i]
$$
where the sum is taken over subsets of $\{1,\ldots,n\}$ and 
where $\Vol_I$ denotes the wedge of $v_i,i\in I$, in the increasing order,
as a multiple of the generator of $\Lambda^{\rk N}N$.
We can use the coefficients of Hessian to build a constant linear function on the solutions of 
$bbGKZ(C^\circ,0)$ as follows.

\begin{proposition}\label{const}
For any solution $\Psi$ of $bbGKZ(C^\circ,0)$ the function of $x=(x_1,\ldots,x_n)$
\begin{equation}\label{with1}
\sum_{d\in C^\circ} {\rm Coeff}_{d}({\rm Hessian}(x)) \Psi_d(x)
\end{equation}
is constant.
\end{proposition}

\begin{proof}
The expression of \eqref{with1} is given by 
$$
\sum_{\vert I\vert =\rk N} \Psi_{\sum_{i\in  I} v_i}(\prod_{i\in I}x_i )\Vol_I^2.
$$
To prove the proposition, we will show that the partial derivative of the above function with respect to $x_1$ vanishes.
This derivative  is given by 
$$
S=\sum_{\vert I\vert =\rk N} \Psi_{v_1+\sum_{i\in  I} v_i}(\prod_{i\in I}x_i )\Vol_I^2
+
\sum_{\vert I\vert =\rk N,I\ni 1} \Psi_{\sum_{i\in  I} v_i}(\prod_{1\neq i\in I}x_i )\Vol_I^2.
$$

\medskip\noindent
For each $I\ni 1$ of size $\rk N$ consider 
the relations \eqref{gkz} for $\mu$ given by composing
$\wedge (\Lambda_{1\neq i\in I})$ with $\Lambda^{\rk N}N\to \ZZ$ and $c=\sum_{i\in I}v_i$.
We will also multiply each relation by $(\prod_{1\neq i\in I} x_i)\Vol_I$.
Notice that this allows us to consider 
$c$ which are not in the interior of $C$, since in that case we multiply 
the fictitious quantities by zero.
When we add these for all $I$ as above we get
$$
\sum_{\vert I \vert = \rk N, I\ni 1} \sum_{j\not\in I\backslash \{1\}} \Psi_{\sum_{i\in I} v_i+v_j}
(\prod_{i\in \{j\}\cup (I\backslash \{1\})} x_i )
\Vol_I 
\Vol_{\{j\}\cup (I\backslash \{1\})} \sigma(I,1,j)
$$
$$
+\sum_{\vert I\vert = \rk N,I\ni 1} \Psi_{\sum_{i\in I} v_i}(\prod_{1\neq i\in I} x_i)\Vol_I^2
=0,$$
where $\sigma(I,1,j)$ is $\pm 1$ depending of whether the location of $j$ in 
$\{j\}\cup (I\backslash \{1\})$ is odd or even. Thus, we have 
$$
S=\sum_{\vert I\vert =\rk N} \Psi_{v_1+\sum_{i\in  I} v_i}(\prod_{i\in I}x_i )\Vol_I^2
-
\sum_{\vert I \vert = \rk N, I\ni 1} \sum_{j\not\in I\backslash \{1\}} \Psi_{\sum_{i\in I} v_i+v_j}
$$$$
(\prod_{i\in \{j\}\cup (I\backslash \{1\})} x_i )
\Vol_I 
\Vol_{\{j\}\cup (I\backslash \{1\})} \sigma(I,1,j)
$$
$$
=\sum_{\vert I\vert =\rk N, I\not\ni 1} \Psi_{v_1+\sum_{i\in  I} v_i}(\prod_{i\in I}x_i )\Vol_I^2
-
\sum_{\vert I \vert = \rk N, I\ni 1} \sum_{j\not\in I} \Psi_{\sum_{i\in I} v_i+v_j}
$$$$
(\prod_{i\in \{j\}\cup (I\backslash \{1\})} x_i )
\Vol_I 
\Vol_{\{j\}\cup (I\backslash \{1\})} \sigma(I,1,j)
$$
$$
=\sum_{\vert J\vert =\rk N, J\not\ni 1} \Psi_{v_1+\sum_{i\in J} v_i }(\prod_{i\in J}x_i)\Vol_J
\sum_{k\in J\cup\{1\}}\Vol_{\{1\}\cup (J\backslash \{k\})}\sigma(J,k,1).
$$
It thus remains to show that for each such $J$ 
\begin{equation}\label{need}
\Vol_J\sum_{k\in J\cup\{1\}}
\Vol_{\{1\}\cup (J\backslash \{k\})}
\sigma(J,k,1)=0.
\end{equation}
If $\Vol_J\neq 0$, then $v_i, i\in J$ span $N_\RR$. As a result, $v_1$ and $v_i, i\in J$
satisfy a linear dependence relation
$$
\sum_{k\in  J\cup \{1\}} \alpha_jv_j=0.
$$
with $\alpha_1\neq 0$.
Since all the elements $v_i$ lie in a hyperplane, it follows that $\sum_{k\in  J\cup \{1\}}\alpha_k=0$.
It remains to observe that 
$$\Vol_{\{1\}\cup (J\backslash \{k\})}\sigma(J,k,1)=
-\Vol_J \alpha_k\alpha_1^{-1}.$$\end{proof}

In view of Proposition \ref{delta}, if one has a pairing that satisfies the 
condition of Conjecture \ref{conj.pairing}, then
plugging in $\Phi_c=\delta_c^0$
$$
\langle \delta_c^0, \Gamma^\circ\rangle 
$$ 
must produce the element of $K_0^c(\PP_\Sigma)$ which is
the contraction of the dual to the Euler characteristics pairing with the rank function.
Consider the element $ch^c[\OO_p]$ of $H_0^c\subseteq \oplus_\gamma H^c_\gamma$ 
which is given
by $\Vol_IF_I$
for any cone $I$ of $\Sigma$ of maximum dimension.  We should view $[\OO_p]$  
as a class of a generic point on $\PP_\Sigma$ (even thought the definition of 
the abelian category only used complexes with cohomology supported on the
$\pi^{-1}(0)$).  Euler characteristics with this $[\OO_p]$ is precisely the rank
function on $K_0(\PP_\Sigma)$. As a consequence,
the following proposition is consistent with Conjecture \ref{conj.pairing}.
\begin{proposition}\label{pairingtest}
For any choice of projective subdivision $\Sigma,$ we have 
$$
\sum_{d\in C^\circ} {\rm Coeff}_{d}({\rm Hessian}(x)) \Gamma^\circ_d(x_1\ldots,x_n)=
\frac{\Vol({\rm conv}(v_1,\ldots,v_n))}
{(2\pi\ii)^{\rk N}}[\OO_p].
$$
where $\Vol$ is the normalized volume.
\end{proposition}

\begin{proof}
By Proposition \ref{const}, the left hand side of the equation is a constant. Therefore, 
it is enough to prove the 
statement asymptotically as $(x_1,\ldots,x_n)$ approaches the  degeneracy point that 
corresponds to the triangulation $\Sigma$. Specifically, let $\mu_i\in \RR$ be chosen 
compatible with $\Sigma$ in the following sense.
We extend these values to a $\Sigma$ piecewise linear function  $\mu$  on $\CC$
by using $\mu(v_i)=\mu_i$ for $\{i\}\in\Sigma$. 
For every positive linear combination $\sum_{j\in J} \alpha_j v_j$ there 
holds
\begin{equation}\label{convex}
\sum_{j\in J} \alpha_j \mu_j\geq \mu(\sum_{j\in J}\alpha_j v_j)
\end{equation}
with equality if and only if $J\in\Sigma$. The fact that such collection of $\mu_i$ 
exists is precisely the projectivity condition on $\Sigma$.

\medskip
 
Consider $x_i=\ee^{t\mu_i}$ and let $t\to +\infty$. We will study the asymptotic behavior
of $\Gamma^\circ$. Uniform convergence allows one to consider limits of individual terms.

\medskip

The terms of  $\Vol_I^2(\prod_{i\in I}x_i)\Gamma^\circ(x_1,\ldots,x_n)_{\sum_{i\in I}v_i}$
for  $(l_j)\in L_{\sum_{i\in I}v_i,\gamma}$ are 
$$
\Vol_I^2(\prod_{i\in I}x_i)\Big(\prod_{j=1}^n\frac {x_j^{l_j+\frac {D_j}{2\pi\ii}}}{\Gamma(1+l_j+\frac {D_j}{2\pi\ii})}\Big)
(\prod_{j, l_j\in \ZZ_{<0}} D_j^{-1} )F_{\{j,l_j\in \ZZ_{<0}\}}.
$$
As $t\to \infty$ this is, up to an invertible element,
$$
\ee^{t(\sum_{i=1}^nl_i\mu_i+\sum_{i\in I}\mu(v_i))}F_{\{j,l_j\in \ZZ_{<0}\}}.
$$

\medskip

In the limit $t\to+\infty$ only  the terms with 
 $\sum_{i=1}^nl_i\mu_i+\sum_{i\in I}\mu_i\geq  0$ contribute.
In addition, to get a nonzero term, we must have ${\{j,l_j\in \ZZ_{<0}\}}$ be a cone in $\Sigma$ that contains
$\sigma(\gamma)$.  

\medskip

Consider the linear relation
$$
\sum_{i=1}^n l_iv_i +\sum_{i\in I}v_i=\sum_{i=1}^n \hat l_i v_i = 0$$
on $v_i$. The negative coefficients of this linear combination are a subset of 
${\{j,l_j\in \ZZ_{<0}\}}\cup \sigma(\gamma)$. Thus we have 
$$
\sum_{\hat l_i\geq 0}\hat  l_i\mu_i \geq \sum_{\hat l_i <0} (-\hat l_i)\mu_i
=\mu( -\sum_{\hat l_i<0}\hat l_i v_i) = \mu(\sum_{\hat l_i\geq 0}\hat  l_iv_i) .
$$
The convexity of $\mu$ then implies that inequality is in fact equality and 
there exists a cone of $\Sigma$ which contains all $v_i$ for which $\hat l_i\geq 0$.
Thus we have that both for $i$ with $l_i\geq 0$ and for $i$ with $\hat l_i<0$ 
there are cones of $\Sigma$ that contain them.  The two sides 
of 
$$
\sum_{i,\hat l_i\geq 0} \hat l_iv_i = \sum_{i,\hat l_i\geq 0}(-\hat l_i)v_i
$$
are contained in, necessarily disjoint cones of $\Sigma$. The only way this can happen 
is when 
all $\hat l_i=0$. This means that we are working in the non twisted sector,
$\gamma=0$. It also means that $l_i=-1$ for $i\in I$ and $l_i=0$ for $i\not\in I$.
Note that $F_I$ is annihilated by all $D_j$ because $\vert I\vert =\rk N$.
Thus the limit of the  corresponding term of the Gamma series is 
$$
\Vol_I^2  \prod_{i\in I}( \frac {1}{\Gamma(\frac {D_i}{2\pi\ii})}D_i^{-1})F_I
=\frac 1{(2\pi\ii)^{\rk N}}\Vol_Ich^c([O_p]).
$$
\medskip

When the above terms  are added over all $I$ that form a basis of a maximum dimensional
cone of $\Sigma$, the claim follows.
\end{proof}

\begin{remark}
Proposition \ref{pairingtest} explains the coefficient $-\frac 3{4\pi^2}$ that we saw in
Section \ref{sec8}. We originally thought that a multiple of 
$$\langle\Phi,\Psi\rangle  = \sum_{c,d} \Phi_c\Psi_d(-1)^{\deg c} {\rm Coeff}_{c+d} (H(x))
$$
would always provide the pairing of Conjecture \ref{conj.pairing}, 
but the example showed this to not be the case. However, Proposition \ref{explicit}
gives hope that some modification of the above guess that takes into account twisted sectors
may work in general.
\end{remark}

\end{document}